\documentclass[12pt,british,a4paper,reqno]{amsart}

        \usepackage[T1]{fontenc}
	\usepackage[utf8]{inputenc}
	\usepackage[british]{babel}

        \usepackage{amsmath}
	\usepackage{amssymb}
	\usepackage{amsthm}
	\usepackage{bbm}
	\usepackage{braket}
        \usepackage[dvipsnames]{xcolor}
	\usepackage[shortlabels]{enumitem}
	\usepackage{fancyhdr}
        \usepackage{float}
	\usepackage{graphicx}
	\usepackage{ifsym}
	\usepackage{indentfirst}
	\usepackage{lmodern}
	\usepackage{mathrsfs}
	\usepackage{mathtools}
	\usepackage{oplotsymbl} %\pentago
	\usepackage{proof}
	\usepackage{qtree}
	\usepackage{scalerel}
	\usepackage{setspace}
	\usepackage{stmaryrd}
	\usepackage{tensor}
	\usepackage{tikz}
	\usepackage{tikz-cd}
	\usepackage{wasysym}
	\usepackage{xparse}
        \usepackage[all, color]{xy}
            \makeatletter
            \def\amsbb{\use@mathgroup \M@U \symAMSb}
            \makeatother
        \usepackage{bbold}

	\usepackage[hyperfootnotes=false,hidelinks]{hyperref}
		\usepackage[capitalise]{cleveref}
		\usepackage{bookmark}
	\usepackage{geometry}

	\bookmarksetup{numbered,open}
	\renewcommand{\geq}{\geqslant}
	\renewcommand{\leq}{\leqslant}
	\renewcommand{\phi}{\varphi}
	\renewcommand{\Im}{\operatorname{Im}\nolimits}
	\renewcommand{\sp}{\mathrm{sp}}

	\providecommand{\corollaryname}{Corollary}
	\providecommand{\definitionname}{Definition}
        \providecommand{\definitionpropname}{Definition-Proposition}
	\providecommand{\examplename}{Example}
	\providecommand{\lemmaname}{Lemma}
	\providecommand{\notationname}{Notation}
	\providecommand{\propositionname}{Proposition}
	\providecommand{\remarkname}{Remark}
	\providecommand{\theoremname}{Theorem}
	\providecommand{\setupname}{Setup}
	\providecommand{\conjecturename}{Conjecture}
	\providecommand{\questionname}{Question}
	\providecommand{\claimname}{Claim}

	\theoremstyle{plain}
		\newtheorem{thm}{\protect\theoremname}[section] 
            \newtheorem*{thm*}{\protect\theoremname}
		\newtheorem{prop}[thm]{\protect\propositionname}
		\newtheorem{lem}[thm]{\protect\lemmaname}

	\theoremstyle{definition}
		\newtheorem{defn}[thm]{\protect\definitionname}
            \newtheorem{defnprop}[thm]{\protect\definitionpropname}

		\newtheorem{notation}[thm]{\protect\notationname}
		\newtheorem{example}[thm]{\protect\examplename}
		\newtheorem{setup}[thm]{\protect\setupname}

	\theoremstyle{remark}
		\newtheorem{rem}[thm]{\protect\remarkname}
		
	\numberwithin{figure}{section}
	\numberwithin{equation}{section}

	\makeatletter
	
		\newcommand\ackname{Acknowledgements}
		\newenvironment{acknowledgements}{%
			\medskip
			\bgroup
			\list{}{\labelwidth\z@
			\leftmargin3pc \rightmargin\leftmargin
			\listparindent\normalparindent \itemindent\z@
			\parsep\z@ \@plus\p@
			
				}%
				\Small
				\item[\hskip\labelsep\scshape\ackname.]%
			}{%
			\endlist\egroup
			}
			
	\makeatother

\let\amph=& %needed for matrix environments in arrow labels
	\usetikzlibrary{matrix,arrows,decorations.pathmorphing,positioning,decorations.pathreplacing}
	\tikzset{commutative diagrams/.cd, 
		mysymbol/.style = {start anchor=center, end anchor = center, draw = none}}
	\tikzcdset{every label/.append style = {font = \footnotesize}}
	\tikzcdset{scale cd/.style={every label/.append style={scale=#1},
    cells={nodes={scale=#1}}}}

	\newcommand{\BA}{\amsbb{A}}

	\newcommand{\BE}{\amsbb{E}}

    \newcommand{\cat}[1]{\mathcal{#1}}
    \newcommand{\fun}[1]{\mathsf{#1}}
    \newcommand{\adj}[1]{\mathscr{#1}}
    \newcommand{\yoneda}{\amsbb{Y}}

		\newcommand{\Ab}{\operatorname{\mathrm{Ab}}\nolimits}

		\newcommand{\op}{\mathrm{op}}
		
		\newcommand{\field}{\amsbb{K}}

		\newcommand{\sse}{\subseteq}

		\newcommand{\Iso}{\operatorname{Iso}\nolimits} %class of isomorphism classes of objects

		\newcommand{\Ker}{\operatorname{Ker}\nolimits}

		\newcommand{\into}{\hookrightarrow}
		\newcommand{\onto}{\rightarrow\mathrel{\mkern-14mu}\rightarrow}

		\newcommand{\im}{\operatorname{im}\nolimits}

		\newcommand{\id}[1]{\mathrm{id}_{#1}}

		\newcommand{\indxx}[1]{\operatorname{\mathrm{index}}_{#1}\nolimits}

		\NewDocumentCommand{\newindx}{O{} m}{\operatorname{\mathrm{ind}}_{#2}^{#1}\nolimits}

		\newcommand{\Ext}{\operatorname{Ext}\nolimits}
		
		\newcommand{\fs}{\mathfrak{s}}
		
		\newcommand{\sus}{\Sigma} % shift/suspension functor notation
            \newcommand{\bbsus}{\mathbb{\Sigma}}

		\newcommand{\rmod}[1]{\operatorname{\mathrm{mod}}\nolimits{#1}}
		
		\newcommand{\rMod}[1]{\operatorname{\mathrm{Mod}}\nolimits{#1}}

	\newcommand{\deff}{\coloneqq}
	
	\newcommand{\ol}[1]{\overline{#1}}

	\newcommand\restr[2]{{\left.\kern-\nulldelimiterspace#1
						\right|_{#2}}}

    \makeatletter
    
    \renewcommand{\andify}{%
		\nxandlist{\unskip, }{\unskip{} \@@and~}{\unskip \penalty-2 \space \@@and~}} 
    \renewcommand\author@andify{%
  		\nxandlist {\unskip ,\penalty-1 \space\ignorespaces}%
		{\unskip {} \@@and~}%
		{\unskip \penalty-2 \space \@@and~}}
    \makeatother

    \let\oldtocsection=\tocsection
    \let\oldtocsubsection=\tocsubsection
    \renewcommand{\tocsection}[2]{\hspace{0em}\oldtocsection{#1}{#2}}
    \renewcommand{\tocsubsection}[2]{\hspace{2em}\oldtocsubsection{#1}{#2}}
\theoremstyle{plain}
\newtheorem{ThmIntro}{Theorem}

\theoremstyle{definition}
\newtheorem{DefIntro}[ThmIntro]{Definition}

\makeatletter
\@namedef{subjclassname@2020}{%
  \textup{2020} Mathematics Subject Classification}
\makeatother
\usepackage{float}
\usepackage[bottom]{footmisc}
\begin{document}

\title{The index with respect to a contravariantly finite subcategory}
\author[Fedele]{Francesca Fedele}
\address{
		School of Mathematics\\
		University of Leeds\\
		Leeds LS2 9JT\\
		UK
	}
    \email{f.fedele@leeds.ac.uk}
\author[J{\o{}}rgensen]{Peter J{\o{}}rgensen}
	\address{
		Department of Mathematics\\
		Aarhus University\\
		8000 Aarhus\\
		Denmark
	}
    \email{peter.jorgensen@math.au.dk}

\author[Shah]{Amit Shah}
    \email{a.shah1728@gmail.com}

\date{\today}



\begin{abstract}

Cluster algebras are categorified by cluster categories, and $g$-vectors are ca\-te\-go\-ri\-fi\-ed by the classic index with respect to cluster tilting subcategories.  However, the recently introduced completed discrete cluster categories of Dynkin type $\BA$ have a very limited supply of cluster tilting subcategories, so we define the index with respect to additive, contravariantly finite subcategories of which there are many more.

This permits us to extend several strong results from the classic theory to completed discrete cluster categories of Dynkin type $\BA$.  Notably, the index with respect to the subcategory generated by a fan triangulation distinguishes between rigid objects.

We also prove that our index is additive on triangles up to an error term.  This extends the key property which permits the classic index to be used in the categorification of cluster algebras.

\end{abstract}

\maketitle

%%%%%%%%%%%%%%%%%%%%%%%%%%%%%%%%%%%%%%%%%%%%%%%%
%%%%%%%%%%%%%%%%%%%%%%%%%%%%%%%%%%%%%%%%%%%%%%%%
%%%%%%%%%%%%%%%%%%%%%%%%%%%%%%%%%%%%%%%%%%%%%%%%
%%%%%%%%%%%%%%%%%%%%%%%%%%%%%%%%%%%%%%%%%%%%%%%%
%%%%%%%%%%%%%%%%%%%%%%%%%%%%%%%%%%%%%%%%%%%%%%%%
%%%%%%%%%%%%%%%%%%%%%%%%%%%%%%%%%%%%%%%%%%%%%%%%
%%%%%%%%%%%%%%%%%%%%%%%%%%%%%%%%%%%%%%%%%%%%%%%%
%%%%%%%%%%%%%%%%%%%%%%%%%%%%%%%%%%%%%%%%%%%%%%%%

\section{Introduction}
\label{sec:intro}

The classic index in triangulated categories is a categorification of $g$-vectors in cluster algebras.  The classic index is defined with respect to a cluster tilting subcategory, but we define the index with respect to an additive, contravariantly finite subcategory.  This extends the definition to a far larger class of subcategories.

The motivation comes from the recently introduced completed discrete cluster categories of Dynkin type $\BA$, which have a very limited supply of cluster tilting subcategories but do have many additive, contravariantly finite subcategories.  
Our definition of the index permits us to extend to these categories several strong results known for the classic index.  

We also extend the key, abstract result that the index is additive on triangles up to an error term.

We refer the reader to \cite[Def.~2.3]{Fomin-Zelevinsky-Cluster-Algebras-I} and \cite[Sec.~6]{FominZelevinsky-Cluster-algebras-IV-Coefficients} for the definition of cluster al\-ge\-bras and $g$-vectors, to \cite[Sec.~2.1]{Palu-Cluster-characters-for-2-Calabi-Yau-triangulated-categories} for the definition of the classic index in triangulated categories, and to \cite[Secs.~4 and 5]{DehyKeller-On-the-combinatorics-of-rigid-objects-in-2-Calabi-Yau-categories} for the explanation of how it categorifies $g$-vectors. 
Completed discrete cluster categories of Dynkin type $\BA$ were introduced in \cite[Sec.~2]{PaquetteYildirim}; the necessary background on these is recalled in Section~\ref{sec:completed-cluster-cat}.
Additional background on the index can be found in \cite{Fedele-Grothendieck}, \cite{Grabowski-Pressland-Cluster-structures}, \cite{Jorgensen-Tropical-friezes-and-the-index-in-higher-homological-algebra}, \cite{JorgensenShah-index}, \cite{Ogawa-Shah-resolution}, \cite{Reid-indecomposable-objects-determined-by-their-index-in-higher-homological-algebra}, \cite{Reid-Tropical}, \cite{Wang-Wei-Zhang}, \cite{WWZ:2023}.

\subsection{The definition of the index}
Let $\cat{C}$ be a skeletally small triangulated category.  Each object $C$ in $\cat{C}$ has a class $[C]^{\sp}$ in the split Grothendieck group $K^{\sp}_{0}(\cat{C})$.
\begin{DefIntro}
\label{def:A}
The {\em index with respect to an additive subcategory $\cat{X}$} of $\cat{C}$ is the map
\[
  \indxx{ \cat{ X }} :
  \operatorname{obj} \cat{C} 
  \to
  K_{0}(\cat{C},\BE_{ \cat{X} },\fs_{ \cat{X} }) 
\]
which maps an object $C$ of $\cat{C}$ to 
\[
  \indxx{ \cat{ X }}(C) \deff [C]_{\cat{X}}.
\]
Here $[C]_{\cat{X}}$ denotes the coset of $[C]^{\sp}$ in the following abelian group.
\[
  K_{0}(\cat{C},\BE_{ \cat{X} },\fs_{ \cat{X} }) 
  \deff 
  K^{\sp}_{0}(\cat{C})
  \Bigg/
  \Bigg\langle
    [A]^{\sp} - [B]^{\sp} + [C]^{\sp}
    \Bigg|\!
    \begin{array}{l}
      \mbox{There is a triangle $A \xrightarrow{} B \xrightarrow{} C \xrightarrow{ c } \Sigma A$} \\[1mm]
      \mbox{with $\cat{C}( X,c ) = 0$ for each $X$ in $\cat{X}$.}
    \end{array}    
  \Bigg\rangle
\]
\end{DefIntro}
\noindent
Note that contravariant finiteness of $\cat{X}$ is not in fact required in the definition but will be necessary for most of our results.  The notation $K_{0}(\cat{C},\BE_{ \cat{X} },\fs_{ \cat{X} })$ is motivated by the existence of an extriangulated category $(\cat{C},\BE_{ \cat{X} },\fs_{ \cat{X} })$; see Section \ref{sec:extriangulated-categories}.
The definition is based on an idea due to \cite[Sec.~2]{PadrolPaluPilaudPlamondon-Associahedra-for-finite-type-cluster-algebras-and-minimal-relations-between-g-vectors} and appeared in \cite[Def.~3.5]{JorgensenShah-index} in the case where $\cat{X}$ is \emph{rigid}, that is, satisfies $\cat{C}( \cat{X},\Sigma\cat{X} ) = 0$.

It was proved in \cite[Prop.~4.11]{PadrolPaluPilaudPlamondon-Associahedra-for-finite-type-cluster-algebras-and-minimal-relations-between-g-vectors} that if $\cat{X}$ is a cluster tilting subcategory of $\cat{C}$, then the definition recovers the classic index in the following sense: 
$K_0( \cat{C},\BE_{\cat{X}},\fs_{\cat{X}} )$ is isomorphic to $K^{\sp}_{0}(\cat{X})$, and under the isomorphism, $\indxx{ \cat{X} }(C)$ corresponds to $[ X_0 ]^{ \sp } - [ X_1 ]^{ \sp }$ when there is a triangle $X_1 \xrightarrow{} X_0 \xrightarrow{} C \xrightarrow{} \Sigma X_1$ in $\cat{C}$ with $X_0$ and $X_1$ in $\cat{X}$.

\subsection{The index in completed discrete cluster categories of Dynkin type \texorpdfstring{$\BA$}{A}}
The completed discrete cluster category $\overline{\mathcal{C}_n}$ of Dynkin type $\BA$ has a combinatorial model $( S,\overline{ M_n } )$ given by a disc $S$ with marked points $M_n$ on the boundary.  The marked points have two-sided convergence to a set of $n$ accumulation points, and adding them to $M_n$ gives the closure $\overline{ M_n }$.  Arcs between points in $\overline{ M_n }$ correspond to indecomposable objects in $\overline{\mathcal{C}_n}$ by \cite[Cor.~3.11]{PaquetteYildirim}.

Figure~\ref{fig:1-acc-pt} shows a nice, so-called fan triangulation of $( S,\overline{ M_n } )$, and cluster theory suggests to think of its additive closure $\cat{X}$ in $\overline{\cat{C}_n}$ as playing the role of a cluster tilting subcategory (see \cite[Thm.~0.5]{GHJ}), but $\cat{X}$ is actually far from cluster tilting; indeed, it is not even rigid.  Hence the classic index does not apply, but our index does.

Indeed, let $\cat{X}$ be the additive closure of a fan triangulation of $( S,\overline{ M_n } )$.  Then the codomain $K_0( \overline{\cat{C}_n},\BE_{\cat{X}},\fs_{\cat{X}} )$ of $\indxx{\cat{X}}$ is the ``correct'' abelian group:
\begin{ThmIntro}
[=\cref{thm:fan-index-isom}]
\label{thm:B}
There is an isomorphism $K_0( \overline{\cat{C}_n},\BE_{\cat{X}},\fs_{\cat{X}} ) \cong K^{\sp}_0( \cat{X} )$.  
\end{ThmIntro}
\noindent
Moreover, our index distinguishes between rigid objects, extending \cite[Thm.~2.3]{DehyKeller-On-the-combinatorics-of-rigid-objects-in-2-Calabi-Yau-categories}.
\begin{ThmIntro}
[=\cref{thm:index_det_rigid}]
\label{thm:C}
There is an injection from the set of isomorphism classes of rigid objects in $\ol{\cat{C}_n}$ to $K_0( \overline{\cat{C}_n},\BE_{\cat{X}},\fs_{\cat{X}} )$ given by $[C] \mapsto \indxx{\cat{X}}(C)$.
\end{ThmIntro}
\noindent
Finally, our index applied to rigid objects gives sign-coherent, linearly independent sets of vectors, extending \cite[Secs.~2.4 and 2.5]{DehyKeller-On-the-combinatorics-of-rigid-objects-in-2-Calabi-Yau-categories}.
\begin{ThmIntro}
[=\cref{prop:rigid}]
\label{thm:D}
If $C\in\ol{\cat{C}_n}$ is rigid and $U$, $V$ are direct summands of $C$, then $\{ \indxx{\cat{X}}(U), \indxx{\cat{X}}(V) \}$ is a sign-coherent subset of $K_0( \overline{\cat{C}_n},\BE_{\cat{X}},\fs_{\cat{X}} )$ with respect to the basis given by the isomorphism in \cref{thm:B}.  
(See \cite[Sec.~2.4]{DehyKeller-On-the-combinatorics-of-rigid-objects-in-2-Calabi-Yau-categories} for the definition of sign-coherence.)
\end{ThmIntro}

\begin{ThmIntro}
[=\cref{prop:linearly-independent}]
\label{thm:E}
If $C\in\ol{\cat{C}_n}$ is rigid and $\cat{U}$ is a finite set of indecomposable, pairwise non-isomorphic direct summands of $C$, then
$
\set{ \indxx{\cat{X}}(U) | U\in\cat{U} }
$
is linearly independent in $K_0( \overline{\cat{C}_n},\BE_{\cat{X}},\fs_{\cat{X}} )$.
\end{ThmIntro}

\subsection{The index is additive up to an error term}
Having defined the index with respect to an additive subcategory and shown that it behaves well in completed discrete cluster categories of Dynkin type $\BA$, we establish the following key, abstract result.  It expresses a property so fundamental it can be viewed as the motivation to use the term {\em index} at all.  The classic case of the theorem is pivotal to the application of the index to the categorification of cluster algebras (see \cite[Prop.~2.2 and Sec.~3.1]{Palu-Cluster-characters-for-2-Calabi-Yau-triangulated-categories}).  

\begin{ThmIntro}
[=\cref{thm:theorem9}]
\label{thm:F}
Let $\cat{X}$ be an additive, contravariantly finite subcategory of the skeletally small idempotent complete triangulated category $\cat{C}$.  
There is a group homomorphism 
\[
  \theta_{\cat{X}} :
  K_{0}( \rmod{\cat{X}} )
  \xrightarrow{}
  K_{0}( \cat{C},\BE_{\cat{X}},\fs_{\cat{X}} )
\]
such that if $A \xrightarrow{} B \xrightarrow{} C \xrightarrow{ c } \Sigma A$ is a triangle in $\cat{C}$, then 
\[
  \indxx{\cat{X}}( A ) - \indxx{\cat{X}}( B ) + \indxx{\cat{X}}( C )
  =
  \theta_{\cat{X}}\big( [\Im\yoneda _{\cat{X}}c] \big).
\]
\end{ThmIntro}
\noindent
Here $\rmod{\cat{X}}$ denotes the abelian category of finitely presented functors $\cat{X}^{\op}\rightarrow \Ab$.  Its Grothendieck group in the classic sense is $K_{0}( \rmod{\cat{X}} )$, and $\yoneda _{\cat{X}}c$ denotes the morphism $\restr{\cat{C}(-,C)}{\cat{X}} \xrightarrow{} \restr{\cat{C}(-,\Sigma A)}{\cat{X}}$ in $\rmod{\cat{X}}$ induced by $C \xrightarrow{ c } \Sigma A$.  The theorem extends \cite[Thm.~A]{JorgensenShah-index} to the case of a non-rigid subcategory $\cat{X}$, but the proof is completely different to that of \cite[Thm.~A]{JorgensenShah-index}.  Indeed, \cref{thm:F} is obtained as a consequence of the following result, which builds on new techniques introduced in \cite[Sec.~3]{CGMZ}.  Note that $K_{0}( \cat{C},\BE_{\cat{X}},\fs_{\cat{X}} )$ is now absent, and only the classic groups $K_{0}( \rmod{\cat{C}} )$ and $K^{\sp}_{0}(\cat{C})$ appear.
\begin{ThmIntro}
[=\cref{thm:theorem5}]
\label{thm:G}
Let $\cat{C}$ be a skeletally small idempotent complete triangulated category.  There is a group homomorphism 
\[
  \theta_{\cat{C}} :
  K_{0}( \rmod{\cat{C}} )
  \xrightarrow{}
  K^{\sp}_{0}(\cat{C})
\]
such that if $A \xrightarrow{} B \xrightarrow{} C \xrightarrow{ c } \Sigma A$ is a triangle in $\cat{C}$, then 
\[
  [ A ]^{\sp} - [ B ]^{\sp} + [ C ]^{\sp}
  =
  \theta_{\cat{C}}\big( [\Im\yoneda _{\cat{C}}c] \big).
\]
\end{ThmIntro}

%%%%%%%%%%%%%%%%%%%%%%%%%%%%%%%%%%%%%%%%%%%%%%%%
%%%%%%%%%%%%%%%%%%%%%%%%%%%%%%%%%%%%%%%%%%%%%%%%
%%%%%%%%%%%%%%%%%%%%%%%%%%%%%%%%%%%%%%%%%%%%%%%%
%%%%%%%%%%%%%%%%%%%%%%%%%%%%%%%%%%%%%%%%%%%%%%%%
%%%%%%%%%%%%%%%%%%%%%%%%%%%%%%%%%%%%%%%%%%%%%%%%
%%%%%%%%%%%%%%%%%%%%%%%%%%%%%%%%%%%%%%%%%%%%%%%%
%%%%%%%%%%%%%%%%%%%%%%%%%%%%%%%%%%%%%%%%%%%%%%%%
%%%%%%%%%%%%%%%%%%%%%%%%%%%%%%%%%%%%%%%%%%%%%%%%

%
\subsection{Conventions}
\label{sec:conventions}

All subcategories are full in this article. 
By an \textit{additive subcategory} of an additive category, we mean a subcategory that is closed under isomorphisms, finite direct sums and direct summands. 
We denote by $\Ab$ the category of all abelian groups.

%%%%%%%%%%%%%%%%%%%%%%%%%%%%%%%%%%%%%%%%%%%%%%%%%%%%%%%%%%%
%%%%%%%%%%%%%%%%%%%%%%%%%%%%%%%%%%%%%%%%%%%%%%%%%%%%%%%%%%%

\section{Extriangulated categories}
\label{sec:extriangulated-categories}

Extriangulated category theory was initiated by Nakaoka--Palu in \cite{NakaokaPalu-extriangulated-categories-hovey-twin-cotorsion-pairs-and-model-structures}. Prototypical classes of examples of these categories are triangulated categories (see \cref{sec:triangulated-is-extriangulated}) and, modulo set-theoretic technicalities, exact categories. 
We briefly recap some details on extriangulated categories, but we refer the reader to \cite[Sec.~2]{NakaokaPalu-extriangulated-categories-hovey-twin-cotorsion-pairs-and-model-structures} for a precise treatment.

%%%%%%%%%%%%%%%%%%%%%%%%%%%%%%%%%%%%%%%%%%%%%
\subsection{Fundamentals}

For this section, let $(\cat{C},\BE,\fs)$ be an extriangulated category. 
Recall that this means $\cat{C}$ is an additive category, $\BE\colon \cat{C}^{\op}\times\cat{C}\to \Ab$ is a biadditive functor and $\fs$ is a realisation of $\BE$. Moreover, this triplet satisfies the axioms laid out in \cite[Def.~2.12]{NakaokaPalu-extriangulated-categories-hovey-twin-cotorsion-pairs-and-model-structures}.

Importantly, the realisation $\fs$ associates to each 
$\alpha\in\BE(C,A)$ an equivalence class 
$\fs(\alpha) 
    = [\begin{tikzcd}[column sep=0.5cm,cramped]
        A \arrow{r}{a}& B\arrow{r}{b} & C
        \end{tikzcd}]
$
of $3$-term sequences in $\cat{C}$.
Recall that two sequences 
$\begin{tikzcd}[column sep=0.5cm,cramped]
    A \arrow{r}{a}& B\arrow{r}{b} & C
\end{tikzcd}$
and 
$\begin{tikzcd}[column sep=0.5cm,cramped]
    A \arrow{r}{a'}& B'\arrow{r}{b'} & C
\end{tikzcd}$
are \emph{equivalent} if there is an isomorphism $f\colon B\to B'$ such that $fa=a'$ and $b'f=b$.
If $\fs(\alpha) 
    = [\begin{tikzcd}[column sep=0.5cm,cramped]
        A \arrow{r}{a}& B\arrow{r}{b} & C
        \end{tikzcd}]
$,
then 
$\begin{tikzcd}[column sep=0.5cm,cramped]
    A \arrow{r}{a}& B\arrow{r}{b} & C
\end{tikzcd}$
is called an \emph{$\fs$-conflation} and 
\begin{equation}\label{eqn:extriangle-notation}
    \begin{tikzcd}
    A \arrow{r}{a}& B\arrow{r}{b} & C \arrow[dashed]{r}{\alpha} &{}
    \end{tikzcd}
\end{equation}
is known as an \emph{$\fs$-triangle}.

%%%%%%%%%%%%%%%%%%%%%%%%%%%%%%%%%%%%%%%%%%%%%

\subsection{Relative theory}
\label{sec:relative-theory}
 
An $\fs$-triangle \eqref{eqn:extriangle-notation} in $(\cat{C},\BE,\fs)$ induces a natural transformation $\alpha_{\sharp}\colon \cat{C}(-,C) \to \BE(-,A)$, given by 
$(\alpha_{\sharp})_{D}(d)
    = d^{*}\alpha
    \deff \BE(d,A)(\alpha)
$ 
for each object $D\in\cat{C}$ and each morphism $d\colon D\to C$ 
(see \cite[Def.~3.1]{NakaokaPalu-extriangulated-categories-hovey-twin-cotorsion-pairs-and-model-structures}).
In addition, any subcategory of $(\cat{C},\BE,\fs)$ induces a relative extriangulated theory using $\alpha_{\sharp}$ by \cite[Sec.~3.2]{HerschendLiuNakaoka-n-exangulated-categories-I-definitions-and-fundamental-properties}.

\begin{defnprop}
\label{def-prop:relative-structure}
    Suppose $\cat{X}$ is a subcategory of $\cat{C}$. 
    Define
    \[
    \BE_{\cat{X}}(C,A) 
    \deff 
    		\Set{ \alpha\in\BE(C,A) | (\alpha_{\sharp})_{X} = 0 \text{ for every } X\in\cat{X}}
    \]
    and put 
    $\fs_{\cat{X}} \deff \restr{\fs}{\BE_{\cat{X}}}$.
    Then the triplet $(\cat{C},\BE_{\cat{X}},\fs_{\cat{X}})$ is an extriangulated category. 
\end{defnprop}

Furthermore, note that if we choose $\cat{X} = \cat{C}$, then the biadditive functor $\BE_{\cat{X}} = \BE_{\cat{C}}$ is trivial, i.e.\ $\BE_{\cat{C}}(C,A) = 0$ for all $A,C\in\cat{C}$. In this case, the extriangulated structure $(\BE_{\cat{C}},\fs_{\cat{C}})$ on $\cat{C}$ corresponds to the split exact structure on $\cat{C}$. 
On the other hand, if $\cat{X} = 0$, then $\BE_{\cat{X}} = \BE_{0} = \BE$ and 
$(\cat{C},\BE_{0},\fs_{0}) = (\cat{C},\BE,\fs)$.

%%%%%%%%%%%%%%%%%%%%%%%%%%%%%%%%%%%%%%%%%%%%%
\subsection{Grothendieck groups}
\label{subsec:Grothendieck_groups}
Suppose now that $(\cat{C},\BE,\fs)$ is a skeletally small extriangulated category. 
Since $\cat{C}$ is an additive category, recall that one defines the \emph{split Grothendieck group} $K^{\sp}_{0}(\cat{C})$ as the quotient of the free abelian group generated by isomorphism classes $[A]$ of objects in $\cat{C}$ by the subgroup generated by elements 
$[A] - [B] + [C]$ for each split exact sequence
\begin{tikzcd}[column sep=0.3cm,cramped]
    A \arrow{r}{}& B\arrow{r}{} & C
\end{tikzcd}
in $\cat{C}$. 
A generating element in $K^{\sp}_{0}(\cat{C})$ is denoted by $[A]^{\sp}$.

Depicting the structure of $(\cat{C},\BE,\fs)$ like in \eqref{eqn:extriangle-notation} makes it apparent that there is a natural definition of the Grothendieck group of $(\cat{C},\BE,\fs)$, which recovers the classical Grothendieck group of a triangulated or an exact category.

\begin{defn}
\label{def:grothendieck-group}
The \emph{Grothendieck group} of $(\cat{C},\BE,\fs)$ is the abelian group  
\[
K_{0}(\cat{C},\BE,\fs) 
    \deff 
    K^{\sp}_{0}(\cat{C})  /  \Braket{ [A]^{\sp} - [B]^{\sp} + [C]^{\sp} | \text{\eqref{eqn:extriangle-notation} is an $\fs$-triangle} }.
\]
\end{defn}

\begin{notation}
    Suppose $\cat{X}$ is a subcategory of $\cat{C}$. We denote the $K_{0}$-class of an object $C\in\cat{C}$ by:
    \begin{enumerate}
        \item $[C]^{\sp}$ in $K^{\sp}_{0}(\cat{C})$; 
        \item $[C]$ in $K_{0}(\cat{C},\BE,\fs)$; and 
        \item $[C]_{\cat{X}}$ in $K_{0}(\cat{C},\BE_{\cat{X}},\fs_{\cat{X}})$.
    \end{enumerate}
\end{notation}

Note that we have a commutative diagram
\begin{equation}\label{eqn:all-the-pi}
\begin{tikzcd}
K^{\sp}_{0}(\cat{C}) = K_{0}(\cat{C},\BE_{\cat{C}},\fs_{\cat{C}})
	\arrow[two heads]{rr}{\pi_{0}}
	\arrow[two heads]{dr}[swap]{\pi_{\cat{X}}}
&& K_{0}(\cat{C},\BE,\fs) = K_{0}(\cat{C},\BE_{0},\fs_{0})\\
& K_{0}(\cat{C},\BE_{\cat{X}},\fs_{\cat{X}})
	\arrow[two heads]{ur}[swap]{\rho_{\cat{X}}}
&
\end{tikzcd}
\end{equation}
of canonical surjections, 
where $\pi_{\cat{X}}([C]^{\sp}) = [C]_{\cat{X}}$ and $\rho_{\cat{X}}([C]_{\cat{X}}) = [C]$ on generators.

%%%%%%%%%%%%%%%%%%%%%%%%%%%%%%%%%%%%%

\subsection{Triangulated categories}
\label{sec:triangulated-is-extriangulated}

We close this preliminary section by explaining some of the above abstract concepts for the specific case of triangulated categories, since this is our focus in the remainder of the article.
Thus, let $(\cat{C},\sus,\triangle)$ be a triangulated category, 
where $\sus$ is the suspension functor and $\triangle$ is the triangulation. 

We obtain an extriangulated category $(\cat{C},\BE,\fs)$ that captures the triangulated structure as follows. 
Define 
$\BE\colon \tensor[]{\cat{C}}{^{\op}}\times\cat{C}\to\Ab$
to be the biadditive functor given by: 
\begin{enumerate}
	\item for objects $A,C\in\cat{C}$, 
\[
\BE(C,A)\deff \cat{C}(C,\sus A);
\]
and

\item 
for morphisms $d\colon D\to C$ and $e\colon A\to E$ in $\cat{C}$, 
the abelian group homomorphism 
$
\BE(d,e)\colon \BE(C,A) \to \BE(D,E)
$
is
\[
\BE(d,e)(c)\deff (\sus e)\circ c\circ d.
\] 
\end{enumerate}

To define a realisation $\fs$ of $\BE$, let $c$ be an element of $\BE(C,A)$. 
Thus, $c\colon C \to \sus A$ is a morphism in $\cat{C}$, and hence is part of a triangle 
\begin{equation}\label{eqn:triangle}
\begin{tikzcd}
A \arrow{r}{a} & B \arrow{r}{b} & C	 \arrow{r}{c} & \sus A.
\end{tikzcd}
\end{equation}
We then put 
$\fs(c) = 
	[ \begin{tikzcd}[column sep=0.5cm,cramped]
	A \arrow{r}{a} & B \arrow{r}{b} & C
	\end{tikzcd} ]$.
The triplet $(\cat{C},\BE,\fs)$ is an extriangulated category by \cite[Prop.~3.22]{NakaokaPalu-extriangulated-categories-hovey-twin-cotorsion-pairs-and-model-structures}. 
In particular, in this case, we have that 
\begin{equation}\label{eqn:s-triangle-for-triangulated-category}
    \begin{tikzcd}
    A \arrow{r}{a}& B\arrow{r}{b} & C \arrow[dashed]{r}{c} &{}
    \end{tikzcd}
\end{equation}
is an $\fs$-triangle in $(\cat{C},\BE,\fs)$ 
if and only if \eqref{eqn:triangle} belongs to $\triangle$.

For an $\fs$-triangle \eqref{eqn:s-triangle-for-triangulated-category}, the natural transformation 
$c_{\sharp}$ defined in \cref{sec:relative-theory} is simply given by  
$(c_{\sharp})_{D}(d)
    = cd \colon D \to \sus A
$
for each object $D\in\cat{C}$ and each morphism $d\colon D\to C$. 
Moreover, 
if $\cat{X}\sse\cat{C}$ is a subcategory, then the relative extriangulated category $(\cat{C},\BE_{\cat{X}},\fs_{\cat{X}})$ satisfies 
\begin{align}
\BE_{\cat{X}}(C,A) 
	&= 
		\Set{ c\in\cat{C}(C,\sus A) | cx = 0 \text{ for every $x\colon X \to C$ with $X\in\cat{X}$} }\nonumber\\
	& = 
		\Set{ c\in\cat{C}(C,\sus A) | \restr{\cat{C}(-,c)}{\cat{X}} = 0}. \label{eqn:relative-yoneda}
\end{align}

Finally, we observe that the Grothendieck group $K_{0}(\cat{C},\BE,\fs)$ of the extriangulated structure on $\cat{C}$ coincides precisely with the classical Grothendieck group 
$K_{0}(\cat{C}) = K_{0}(\cat{C},\sus,\triangle)$ of the triangulated category.

%%%%%%%%%%%%%%%%%%%%%%%%%%%%%%%%%%%%%%%%%%%%%%%%
%%%%%%%%%%%%%%%%%%%%%%%%%%%%%%%%%%%%%%%%%%%%%%%%

\section{The discrete cluster category of type \texorpdfstring{$\BA$}{A} and its completion}

In this section, we recall the construction and main properties of Igusa--Todorov's discrete cluster categories of type $\BA$ \cite{IgusaTodorov} and their completions by Paquette--Y{\i}ld{\i}r{\i}m \cite{PaquetteYildirim}. 
In \cref{sec:index-in-cluster-cats}, we study our index for certain non-rigid subcategories of completed discrete cluster categories.

%%%%%%%%%%%%
\subsection{The discrete cluster category \texorpdfstring{$\mathcal{C}_n$}{Cn}}\label{sec:discrete-cluster-cat}

Let $S$ be the disc, oriented anticlockwise, and consider a discrete set $M_n$ of infinitely many marked points on its boundary circle $\partial S$, having a positive number $n$ of two-sided accumulation points. Note that the accumulation points are not included in the set  $M_n$. The corresponding \emph{discrete cluster category of type $\BA$} over a fixed algebraically closed field $\field$, denoted $\mathcal{C}_n$, was first introduced by Holm--J{\o}rgensen in \cite{holmjorgensen2012cluster} for $n=1$ and then by Igusa--Todorov in \cite{IgusaTodorov} for general $n$. This is a small, $\field$-linear, Hom-finite, Krull-Schmidt, $2$-Calabi-Yau, triangulated category, with suspension denoted $\Sigma$. As explained in more detail in \cite{IgusaTodorov}, essential properties of $\mathcal{C}_n$ can be described using the geometric model $(S,M_n)$:
\begin{itemize}
    \item There is a bijection between the isomorphism classes of indecomposable objects and the arcs of $(S,M_n)$ between non-neighbouring points in $M_n$.(Arcs are always considered up to homotopy, whilst fixing the marked points.)
    \item Identifying an indecomposable $C\in\mathcal{C}_n$ with the corresponding arc $C=\{c_0,c_1\}$, the suspension $\sus C$ of $C$ is obtained by moving the endpoints of the arc by one clockwise step.
    \item For indecomposable objects $A,C\in\mathcal{C}_n$, we have $\mathcal{C}_n(C,\Sigma A) = \Ext^1_{\mathcal{C}_n}(C,A)$ is isomorphic to $\field$ if and only if 
    the arcs corresponding to $A$ and $C$ cross transversely (in the interior of $S$), and to $0$ otherwise.
\end{itemize}
Using the above, all the morphisms and triangles, including the Auslander-Reiten triangles, can be described.

As mentioned in \cref{sec:intro}, the triangulated index was originally introduced by Palu with respect to cluster tilting subcategories (see \cite{Palu-Cluster-characters-for-2-Calabi-Yau-triangulated-categories}). For discrete cluster categories of type $\BA$, cluster tilting subcategories have been fully classified, first by Liu--Paquette in \cite{LP} for $n=1,2$, and later by Gratz--Holm--J{\o}rgensen in \cite{GHJ} for general $n\geq 1$. In the following, by a \emph{triangulation} we mean a maximal set of arcs that pairwise do not cross transversely; see \cite[Def.~0.4]{GHJ} for the definition of a fountain and a leapfrog. 

\begin{prop}[{\cite[Thm.~0.5]{GHJ}}]\label{thm:CT-subcats-for-Cn}
    Let $\mathcal{X}$ be a set of arcs in $(S,M_n)$. Then the subcategory of $\mathcal{C}_n$ additively generated by $\mathcal{X}$ is a cluster tilting subcategory if and only if $\mathcal{X}$ is a triangulation  of $(S,M_n)$ such that each accumulation point has a fountain or a leapfrog converging to it.
\end{prop}

%%%%%%%%%%%%%%%%%%%%%%%%%%%%%%%%%%%%%%%%%%%%%%%%
\subsection{The completed discrete cluster category \texorpdfstring{$\overline{\mathcal{C}_n}$}{Cn}}
\label{sec:completed-cluster-cat}

In the geometric model for the category $\mathcal{C}_n$,  accumulation points are not elements of the set of marked points, and hence the arcs at the accumulation points do not correspond to objects in $\mathcal{C}_n$. On the other hand, in the completed version, introduced for one accumulation point by Fisher in \cite{Fisher} and for $n$ accumulation points by Paquette--Y{\i}ld{\i}r{\i}m in \cite{PaquetteYildirim}, accumulation points become marked points and arcs leaving/entering them become objects in the corresponding ``completed'' cluster category $\overline{\mathcal{C}_n}$. We briefly recall the construction from \cite{PaquetteYildirim}.

Consider $(S,M_{2n})$ as in \cref{sec:discrete-cluster-cat}, and number the intervals between the $2n$ accumulation points from $1$ to $2n$ anticlockwise.
The \emph{completed discrete cluster category of type $\BA_{n}$} is defined as the Verdier quotient $\overline{\mathcal{C}_n} \deff \mathcal{C}_{2n}/\cat{D}$, where $\cat{D}$ is the thick subcategory of $\mathcal{C}_{2n}$ additively generated by all arcs with both endpoints in the same even-numbered interval. 
Note that 
$\cat{D}^\perp
	= \Set{ C\in \mathcal{C}_{2n} | \mathcal{C}_{2n}(\cat{D},C) = 0  }
$ 
then corresponds to all arcs with both endpoints in odd-numbered intervals. 
The category $\overline{\mathcal{C}_n}$ is still a small, $\field$-linear, Hom-finite, Krull-Schmidt triangulated category, with suspension denoted again $\Sigma$. However, it fails not only to be $2$-Calabi-Yau but also to have a Serre functor at all.\footnote{This was first noted in \cite[p.~56]{PaquetteYildirim}. By \cite[Thm.~I.2.4]{ReitenVandenBergh:2002}, $\overline{\mathcal{C}_n}$ cannot admit a Serre functor since it does not have Auslander-Reiten triangles in general (see \cite[Footnote 1 in Fig.~4]{CKP}).}

Passing to the Verdier quotient corresponds to collapsing each even-numbered interval to an accumulation point in the geometric model $(S,M_{2n})$. Thus, the corresponding geometric model of $\overline{\mathcal{C}_n}$ is 
$(S,\overline{M_n})$, where $\overline{M_n}$ is the infinite discrete set $M_n$ as above together with the $n$ accumulation points; see properties \ref{completion-objects}--\ref{morph_to_C} below. 
In the Verdier quotient: 
arcs in $\cat{D}$ vanish; 
arcs in $\cat{D}^\perp$ correspond to arcs between non-accumulation points; 
while arcs outside $\cat{D}\cup\cat{D}^\perp$ correspond to \textit{limit arcs}, that is, arcs where at least one endpoint is an accumulation point. Note also that each such limit arc is the image of infinitely many arcs under the Verdier quotient map $\mathcal{C}_{2n}\rightarrow \overline{\mathcal{C}_n}$.

Before recalling the fundamental properties of $\overline{\mathcal{C}_n}$, we set up some notation.

\begin{notation}\label{notation:discrete-cluster}
The boundary circle $\partial S$ inherits an anticlockwise orientation from the orientation on $S$. This orientation gives a cyclic ordering on $\overline{M_n}$, to which we implicitly refer in the rest.
  \begin{enumerate}[label={\textup{(\roman*)}}]
      \item\label{successor-predecessor} Each marked point $c \in \overline{M_n}$ that is not an accumulation point has an (immediate) successor $c^+$ in $\overline{M_n}$, which is the next marked point anticlockwise from $c$. 
      Similarly, $c$ has a predecessor $c^-\in \overline{M_n}$ that is the next marked point clockwise from $c$. 
      Formally, accumulation points do not have a direct successor or predecessor, but to simplify arguments we use the convention that if $c$ is an accumulation point, then $c^+ = c = c^-$.
      \item\label{ordering} Given points $x,y \in \overline{M_n}$, there is a total order on the interval $[x,y] \sse \overline{M_n}$ of points anticlockwise from $x$ to $y$ including $x$ and $y$. Moreover, we denote by $(x,y),\, [x,y)$ and $(x,y]$ the intervals of marked points between $x$ and $y$ excluding, respectively, $x$ and $y$, only $y$, and only $x$.
      \item In our figures below, black circles indicate accumulation points, while dashes indicate points that could be accumulation points or not (unless specified).
  \end{enumerate}  
\end{notation}

In the following list, we recall some important properties of $\overline{\mathcal{C}_n}$. See \cite[Cor.~3.11, Prop.~3.4, Prop.~3.14 and the discussion thereafter]{PaquetteYildirim} and \cite[Prop.~4.10]{Franchini} for more details. 
\smallskip

\begin{enumerate}[label={\textup{(P\arabic*)}}]
    \item\label{completion-objects} There is a bijection between the isomorphism classes of indecomposable objects in $\overline{\mathcal{C}_n}$ and the arcs of $(S,\overline{M_n})$ between non-neighbouring points in $\overline{M_n}$. We identify an indecomposable object $C$ and the corresponding arc $\{c_0,c_1\}$. 
    Furthermore, by abuse of notation, we denote the subcategory of $\overline{\cat{C}_{n}}$ additively generated by a set $\cat{X}$ of arcs in $(S,\overline{M_n})$ also by $\cat{X}$. 
    
    \item\label{suspension} 
    The suspension of $C=\{c_0,c_1\}$ is obtained by moving the endpoints of the arc by one \emph{clockwise} step when possible, that is, $\Sigma C=\{ c_0^-,c_1^-\}$ (see \cref{notation:discrete-cluster}\ref{successor-predecessor}). 
    \item\label{homspaces} For indecomposable objects $A,C\in\overline{\mathcal{C}_n}$, we have $\overline{\mathcal{C}_n}(C,\Sigma A) = \Ext^1_{\overline{\mathcal{C}_n}}(C, A)\cong \field$ if and only if one of the following holds: 
    \begin{enumerate}[label={\textup{(\roman*)}}]
        \item\label{transverse} the arcs $A$ and $C$ cross transversely;
        \item\label{one-endpoint} the arcs $A$ and $C$ are distinct, share exactly one accumulation point $p$ as an endpoint, and we can move the other endpoint $c_0 \neq p$ of $C$ anticlockwise along $\partial S$ to the other endpoint $a_0\neq p$ of $A$ \emph{without} passing through $p$ itself; or 
        \item\label{equal} the arcs $A$ and $C$ are equal, and both endpoints of $A=C$ are accumulation points.
    \end{enumerate}
    Otherwise, $\overline{\mathcal{C}_n}(C,\Sigma A) =  \Ext^1_{\overline{\mathcal{C}_n}}(C, A)=0$.
    \item\label{triangles-transverse} In \ref{homspaces}\ref{transverse}, 
    the non-zero morphisms in $\overline{\mathcal{C}_n}(C,\Sigma A)$ and $\overline{\mathcal{C}_n}(A,\Sigma C)$ extend, respectively, to triangles
\[
\begin{tikzcd}[column sep=0.5cm]
A \arrow{r}
& B_1\oplus B_2 \arrow{r}
& C \arrow{r}
& \sus A
\end{tikzcd}
	\hspace{0.5cm} \text{and} \hspace{0.5cm}
\begin{tikzcd}[column sep=0.5cm]
C \arrow{r}
& D_1\oplus D_2 \arrow{r}
& A \arrow{r}
& \sus C,
\end{tikzcd}
\]
    where $B_i,\,D_i$ are illustrated in the left picture of Figure~\ref{fig:figtriangleIV} and are equal to $0$ if they are arcs between neighbouring marked points.

    \item\label{triangle-one-acc-pt}
    In \ref{homspaces}\ref{one-endpoint}, 
    the non-zero morphism in $\overline{\mathcal{C}_n}(C,\Sigma A)$ extends to the triangle
\[
\begin{tikzcd}[column sep=0.5cm]
A \arrow{r}
& E \arrow{r}
& C \arrow{r}
& \sus A,
\end{tikzcd}
\]
    where $E$ is the arc illustrated in the right picture of Figure~\ref{fig:figtriangleIV} and is equal to $0$ if $\Sigma A=C$. On the other hand, note that $\overline{\mathcal{C}_n}(A,\Sigma C)=0$ and \emph{there is no} non-trivial triangle 
    of the form 
    $C\rightarrow F\rightarrow A\rightarrow \Sigma C.$
    
    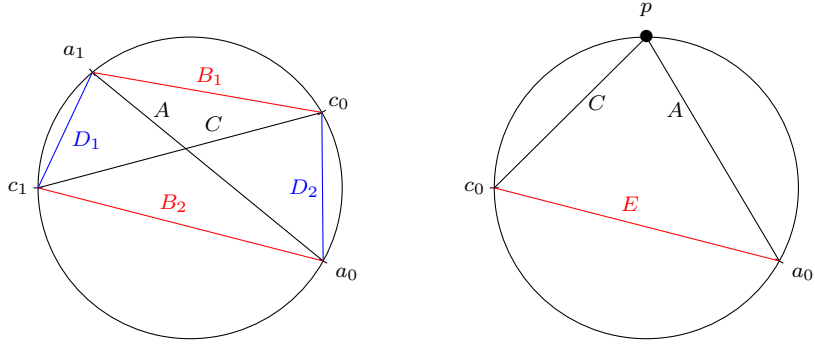
\begin{figure}[ht]
  \centering
    \begin{tikzpicture}[scale=2]
    \begin{scope}
    \draw (0,0) circle (1cm);

      %endvertices of C
      \draw (30:0.97cm) -- (30:1.03cm);
      \draw (30:1.13cm) node{$\scriptstyle c_{0}$};
      
      \draw (180:0.97cm) -- (180:1.03cm);
      \draw (180:1.13cm) node{$\scriptstyle c_{1}$};
      
      %endvertices of A
      \draw (-29:0.97cm) -- (-29:1.03);
      \draw (-29:1.18cm) node{$\scriptstyle a_{0}$};

      \draw (130:0.97cm) -- (130:1.03cm); 
      \draw (130:1.18cm) node{$\scriptstyle a_{1}$};
      
      %arcs
      \draw (180:1cm) -- (30:1cm);
      \draw (-29:1cm) -- (130:1cm);
     
     % red arcs
     \draw [red] (-29:1cm) -- (180:1cm);
     \draw [red] (30:1cm) -- (130:1cm);
     % blue arcs
     \draw [blue] (130:1cm) -- (180:1cm);
     \draw [blue] (30:1cm) -- (-29:1cm);
     
     %labels
     \draw  (110:0.55cm) node{$\scriptstyle A$};
     \draw  (70:0.45cm) node{$\scriptstyle C$};
     \draw [red] (80:0.75cm) node{$\scriptstyle B_{1}$};
     \draw [red] (-136:0.15cm) node{$\scriptstyle B_{2}$};
     \draw [blue] (155:0.75cm) node{$\scriptstyle D_{1}$};
      \draw [blue] (0:0.75cm) node{$\scriptstyle D_{2}$};
      \end{scope}
      \begin{scope}[xshift=3cm]
      \draw (0,0) circle (1cm);

      \draw (180:0.97cm) -- (180:1.03cm);
      \draw (180:1.13cm) node{$\scriptstyle c_{0}$};
      
      \draw (-29:0.97cm) -- (-29:1.03);
      \draw (-29:1.18cm) node{$\scriptstyle a_{0}$};

      \draw (90:1cm) node{$\bullet$}; 
      \draw (90:1.18cm) node{$\scriptstyle p$};
      
      %arcs
      \draw (90:1cm) -- (180:1cm);
      \draw (-29:1cm) -- (90:1cm);
     
     % red arcs
     \draw [red] (-29:1cm) -- (180:1cm);
     
     %labels
     \draw  (120:0.65cm) node{$\scriptstyle C$};
     \draw  (70:0.55cm) node{$\scriptstyle A$};
     \draw [red] (-136:0.15cm) node{$\scriptstyle E$};
      \end{scope}
    \end{tikzpicture}
    \caption{Terms appearing in nontrivial triangles with indecomposable endterms in $\overline{\mathcal{C}_n}$. Note that $p$ is an accumulation point, while the rest of the indicated points may be either accumulation points or not.}
    \label{fig:figtriangleIV}
\end{figure}
    
    \item In \ref{homspaces}\ref{equal}, 
	we have 
    $\Sigma A=A=C=\Sigma C$ and the triangle extending the non-zero morphism in $\overline{\mathcal{C}_n}(C,\Sigma A) = \overline{\mathcal{C}_n}(C,C)$ is
\[
\begin{tikzcd}[column sep=0.5cm]
C \arrow{r}
& 0 \arrow{r}
& C \arrow{r}{\cong}
& C.
\end{tikzcd}
\]
   
\item\label{morph_from_C} Let $C=\{c_0,c_1\}$ be an indecomposable in $\overline{\mathcal{C}_n}$. An indecomposable $D=\{d_0,d_1\}$ in $\overline{\mathcal{C}_n}$
(where $d_0\in [c_0 ,c_1]$)  
is such that $\overline{\mathcal{C}_n}(C,D)\cong \field$ if and only if 
$d_0\in[c_0,c_1^-)$ and $d_1\in[c_1,c_0^-)$.
Moreover, for such a $D$, the indecomposables $S=\{s_0,s_1\}$ such that a non-zero morphism $C\rightarrow D$ factors through $S$ are exactly those having one endpoint in $[c_0,d_0]$ and the other in $[c_1,d_1]$. See the left picture in Figure~\ref{fig:fig_morphisms_from_to_x} for an illustration.

\item\label{morph_to_C} Let $C=\{c_0,c_1\}$ be an indecomposable in $\overline{\mathcal{C}_n}$. An indecomposable $B=\{b_0,b_1\}$ in $\overline{\mathcal{C}_n}$ 
(where $b_0\in (c_0 ,c_1]$) 
is such that $\overline{\mathcal{C}_n}(B,C)\cong \field$ if and only if 
$b_0\in(c_0^+,c_1]$ and $b_1\in(c_1^+,c_0]$. 
Moreover, for such a $B$, the indecomposables $S=\{s_0,s_1\}$ such that a non-zero morphism $B\rightarrow C$ factors through $S$ are exactly those having one endpoint in $[b_0,c_1]$ and the other in $[b_1,c_0]$. See the right picture in Figure~\ref{fig:fig_morphisms_from_to_x} for an illustration. 
\end{enumerate}
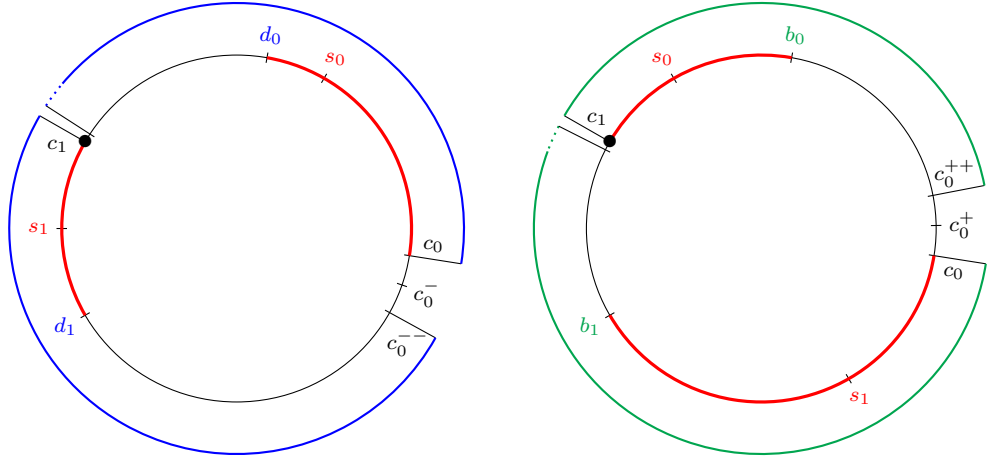
\begin{figure}[ht]
\centering
\begin{tikzpicture}[scale=2.3]
\begin{scope}
   \draw (0,0) circle (1cm);
   
   %endvertices of x
   \draw (-9:0.97cm) -- (-9:1.3cm);
   \draw (-5:1.13cm) node{$\scriptstyle c_{0}$};
   \draw (150:0.97cm) -- (150:1.3cm);
   \draw (155:1.13cm) node{$\scriptstyle c_{1}$};
   %\draw (150:1cm) -- (-9:1cm);
   
   %c
   \draw (-19:0.97cm) -- (-19:1.03);
   \draw (-19:1.15cm) node{$\scriptstyle c_{0}^{-}$}; 
   %\draw (140:0.97cm) -- (140:1.03cm); 
   %\draw (140:1.15cm) node{$\scriptstyle x_{1}^{-}$};
   \draw (-29:0.97cm) -- (-29:1.3);
   \draw (-34:1.18cm) node{$\scriptstyle c_{0}^{--}$};
   %\draw (130:0.97cm) -- (130:1.3cm); 
   %\draw (123:1.11cm) node{$\scriptstyle x_{1}^{--}$};
   
   %d
   \draw (80:0.97cm) -- (80:1.03cm);
   \draw [blue] (80:1.13cm) node{$\scriptstyle d_{0}$};
   \draw (-150:0.97cm) -- (-150:1.03cm);
   \draw [blue] (-150:1.13cm) node{$\scriptstyle d_{1}$};
   
   %s
   \draw (60:0.97cm) -- (60:1.03);
   \draw [red] (60:1.13cm) node{$\scriptstyle s_{0}$};
   \draw (180:0.97cm) -- (180:1.03);
   \draw [red] (180:1.13cm) node{$\scriptstyle s_{1}$}; 
   
   %red arcs
   \draw[very thick, red] ([shift=(80:1cm)]0,0) arc (80:-9:1cm);
   \draw[very thick, red] ([shift=(-210:1cm)]0,0) arc (-210:-150:1cm);   
   %acc point
   \draw (150:1cm) node{$\bullet$};
   
   %blue arcs
    \draw (147:0.97cm) -- (147:1.3);
    
%    %arcs 
%      \draw [color=red] (180:1cm) -- (60:1cm); % S
%      \draw (-9:1) -- (150:1cm); % C 
%	  \draw [color=blue!75!black] (-150:1) -- (80:1cm); % D
  
   \draw[thick, blue] ([shift=(140:1.3cm)]0,0) arc (140:-9:1.3cm);
   \draw[thick, dotted, blue] ([shift=(147:1.3cm)]0,0) arc (147:140:1.3cm);
   \draw[thick, blue] ([shift=(-210:1.3cm)]0,0) arc (-210:-29:1.3cm);
\end{scope}
\begin{scope}[xshift=3cm]
    \draw (0,0) circle (1cm);
   
   %endvertices of x and x
   \draw (-9:0.97cm) -- (-9:1.3cm);
   \draw (-13:1.13cm) node{$\scriptstyle c_{0}$};
   \draw (150:0.97cm) -- (150:1.3cm);
   \draw (146:1.13cm) node{$\scriptstyle c_{1}$};
   %\draw (150:1cm) -- (-9:1cm);
   
   %x0+,x0++,x1+,x1++
   \draw (1:0.97cm) -- (1:1.03);
   \draw (1:1.15cm) node{$\scriptstyle c_{0}^{+}$}; 
   %\draw (160:0.97cm) -- (160:1.03cm); 
   %\draw (160:1.15cm) node{$\scriptstyle x_{1}^{+}$};
   \draw (11:0.97cm) -- (11:1.3);
   \draw (16:1.15cm) node{$\scriptstyle c_{0}^{++}$};
   %\draw (170:0.97cm) -- (170:1.3cm); 
   %\draw (176:1.11cm) node{$\scriptstyle x_{1}^{++}$};
   
   %z
   \draw (80:0.97cm) -- (80:1.03cm);
   \draw [Green] (80:1.13cm) node{$\scriptstyle b_{0}$};
   \draw (-150:0.97cm) -- (-150:1.03cm);
   \draw [Green] (-150:1.13cm) node{$\scriptstyle b_{1}$};

   %green arcs
  
   \draw[thick, Green] ([shift=(150:1.3cm)]0,0) arc (150:11:1.3cm);
   \draw[thick, Green] ([shift=(160:1.3cm)]0,0) arc (160:351:1.3cm);
   \draw[thick, dotted, Green] ([shift=(153:1.3cm)]0,0) arc (153:160:1.3cm);
   \draw (153:0.97cm) -- (153:1.3);
   
   %s
   \draw (-60:0.97cm) -- (-60:1.03);
   \draw [red] (-60:1.13cm) node{$\scriptstyle s_{1}$};
   \draw (120:0.97cm) -- (120:1.03);
   \draw [red] (120:1.13cm) node{$\scriptstyle s_{0}$}; 
   
   %red arcs
   \draw[very thick, red] ([shift=(-9:1cm)]0,0) arc (-9:-150:1cm);
   \draw[very thick, red] ([shift=(80:1cm)]0,0) arc (80:150:1cm); 
   %acc point
   \draw (150:1cm) node{$\bullet$};
\end{scope}
\end{tikzpicture} 
\caption{Morphism $C\rightarrow D$ on the left and $B\rightarrow C$ on the right, and their factorisations through $S$. In the illustrated case, $c_1$ is an accumulation point, while $c_0$ is not. 
The blue arcs in the left picture indicate the different intervals for the endpoints of $D$; 
the green arcs in the right picture  indicate the different intervals for the endpoints of $B$. 
The dots in the coloured arcs indicate the interval approaches $c_1$ without reaching it.}
\label{fig:fig_morphisms_from_to_x}
\end{figure}

Analogously to $\mathcal{C}_n$ (see \cref{thm:CT-subcats-for-Cn}), cluster tilting subcategories of $\overline{\mathcal{C}_n}$ correspond to certain triangulations of $(S,\overline{M_n})$. 
By a \emph{triangulation} we again mean a maximal set of arcs that pairwise do not cross transversely.
In the following, note that $\cat{C}_n$ can be seen as a  triangulated subcategory of $\overline{\cat{C}_n}$.

\begin{prop}[{\cite[Thm.~4.4]{PaquetteYildirim}}]
\label{prop_ct_PY}
The additive subcategory $\cat{X}\sse\overline{\mathcal{C}_n}$ given by a set $\cat{X}$ of arcs in $(S,\overline{M_n})$ 
is cluster tilting if and only if the following conditions hold:
   \begin{enumerate}
       \item $\mathcal{X}\cap \cat{C}_n$ is a triangulation of $(S, M_n)$;
       
       \item there is a two-sided fountain (in the sense of \cite[Def.~4.3(3)]{PaquetteYildirim}) in $\mathcal{X}\cap \cat{C}_n$ at every accumulation point; and 
       
       \item the subcategory of $\overline{\mathcal{C}_n}$ additively generated by $\mathcal{X}$ is the geometric completion of $\mathcal{X}\cap \cat{C}_n$, that is, the additive subcategory of $\overline{\mathcal{C}_n}$ generated by the arcs in $\mathcal{X}\cap \cat{C}_n$ plus those corresponding to accumulation arcs of $\mathcal{X}\cap \cat{C}_n$.
   \end{enumerate}
   In other words, cluster tilting subcategories of $\overline{\mathcal{C}_n}$ correspond to triangulations of $(S, M_n)$ that have two-sided fountain at each accumulation point, plus one limit arc for every two-sided fountain. 
\end{prop}

%%%%%%%%%%%%%%%%%%%%%%%%%%%%%%%%%%%%%%%%%%%%%%%%%%%%%%%%%%%%%
%%%%%%%%%%%%%%%%%%%%%%%%%%%%%%%%%%%%%%%%%%%%%%%%%%%%%%%%%%%%%
\section{The index in the completed discrete cluster category}
\label{sec:index-in-cluster-cats}

A substantial difference between $\mathcal{C}_n$ and $\ol{\cat{C}_n}$ is that many triangulations in $(S,\overline{M_n})$ correspond to subcategories of $\overline{\mathcal{C}_n}$ that are not rigid, and hence also not cluster tilting. 
For instance, two arcs meeting at an accumulation point do not cross transversely, and so may well be part of a triangulation (see \cref{exam:1-accum-pt,exam:2-accum-pt}), but the corresponding subcategory will not be rigid by \ref{homspaces}. 
This lack of rigidity is not a problem for us, however, since the index we introduce in \cref{def:A} is defined with respect to \emph{any} contravariantly finite, additive subcategory.

For a fan triangulation $\cat{X}$ (see \cref{def:fan}) of $(S,\overline{M_n})$, we prove 
over Sections~\ref{sec:index-for-fan}--\ref{sec:mutation} that 
the index with respect to $\cat{X}$ 
behaves much like the classical index. In particular, the index 
takes values in $K_{0}^{\sp}(\cat{X})$ (\cref{thm:fan-index-isom}) and 
determines rigid objects (\cref{thm:index_det_rigid}).
The index on rigid objects satisfies sign-coherence and linear independence properties (\cref{prop:rigid} and \cref{prop:linearly-independent}). 
Moreover, if $\cat{X}$ is also rigid, then we examine how the index behaves under mutation (\cref{thm:mutation}). 
\cref{sec:examples} contains examples.

%%%%%%%%%%%%%%%%%%%%%%%%%%%%%%%%%%%%%%%%%%%%%%%%%%%%%%%%%%%%%

\subsection{Fan triangulations}
\label{sec:fan-triangulations}

The proofs of our main results in \cref{sec:index-in-cluster-cats} rely on having a triangulation $\cat{X}$ of $(S,\overline{M_n})$ that gives a contravariantly finite subcategory (rigid or not) in $\ol{\cat{C}_n}$. 
As we will show in this subsection, these kinds of  triangulations must be fan triangulations, which have already played an important role in e.g.\ \cite{CF,CKP,GHJ,PaquetteYildirim}.

\begin{defn}[{\cite[Defs.~2.16 and 3.30]{CF}}, {\cite[Def.~4.1]{CKP}}, {\cite[Def.~4.3(4)]{PaquetteYildirim}}]
\label{def:fan}
    In $(S,\overline{M_n})$, an \emph{infinite leapfrog} 
    consists of two sets $\{A_i\}_{i\in I}$ and $\{B_i\}_{i\in I}$ of arcs, indexed by a set $I\in\Set{\mathbb{Z}, \mathbb{Z}_{\leq 0},\mathbb{Z}_{\geq 0}}$, such that: 
    \begin{itemize}
        \item all arcs from $\{A_i,B_i\}_{i\in I}$ are distinct and pairwise do not cross;
        \item each $A_i$ is incident with one endpoint of $B_i$ and, if $i-1\in I$, also one endpoint of $B_{i-1}$;
        \item each $B_i$ is incident with one endpoint of $A_i$ and, if $i+1\in I$, also one endpoint of $A_{i+1}$; and 
        \item there is a curve $\gamma$ in $S$ (not necessarily between marked points) that crosses all the arcs $A_i$ and $B_i$. 
    \end{itemize}
    A \emph{fan triangulation} of $(S,\overline{M_n})$ is a triangulation with no infinite leapfrogs. 
    See \cite[Figs.~7 and 8]{CKP} for illustrations of these concepts.
\end{defn}

The following result uses Franchini's full classification \cite[Thm.~5.4]{Franchini} of contravariantly (resp.\ covariantly) finite subcategories in $\ol{\cat{C}_n}$ and a nice result \cite[Prop.~4.5]{CKP} of \c{C}anak\c{c}\i--Kalck--Pressland on fan triangulations. 

\begin{lem}
\label{lem:precovering-triangulation-is-fan}

Suppose $\cat{X}$ is a triangulation of $(S,\overline{M_n})$. Then the following are equivalent.
\begin{enumerate}[label={(\arabic*)}]
    \item\label{item:func-finite} $\cat{X}\sse\ol{\cat{C}_n}$ is functorially finite.
    \item\label{item:contra-finite} $\cat{X}\sse\ol{\cat{C}_n}$ is contravariantly finite.
    \item\label{item:co-finite} $\cat{X}\sse\ol{\cat{C}_n}$ is covariantly finite.
    \item\label{item:fan} $\cat{X}$ is  a fan triangulation.
\end{enumerate}
\end{lem}

\begin{proof}
We prove only \ref{item:contra-finite} $\Leftrightarrow$ \ref{item:fan}; the equivalence \ref{item:co-finite} $\Leftrightarrow$ \ref{item:fan} is dual. It follows that \ref{item:func-finite} and \ref{item:fan} are equivalent.

Assume first that the triangulation $\cat{X}$ gives a contravariantly finite subcategory $\cat{X}$. 
By \cite[Thm.~5.4]{Franchini}, we know $\cat{X}$ satisfies the completed precovering conditions 
$(\ol{\mathrm{PC}}1)$,
$(\ol{\mathrm{PC}}2)$, 
$(\ol{\mathrm{PC}}2')$, 
$(\ol{\mathrm{PC}}3)$ and 
$(\ol{\mathrm{PC}}3')$ 
given in \cite[Def.~5.3]{Franchini}. 
If $\cat{X}$ were to contain an infinite leapfrog, then $(\ol{\mathrm{PC}}2)$ implies $\cat{X}$ would contain arcs that cross transversely, which is impossible as $\cat{X}$ is a triangulation. 
Thus, $\cat{X}$ is a fan triangulation.

The converse is shown in \cite[the proof of Thm.~4.6]{CKP}; it can also be done directly using the completed precovering conditions. 
\end{proof}

%%%%%%%%%%%%%%%%%%%%%%%
\subsection{The index with respect to a fan triangulation}
\label{sec:index-for-fan}

Throughout \cref{sec:index-for-fan,sec:index-determines-rigids,sec:signs-independence-basis,sec:mutation}, we assume the following.

\begin{setup}\label{setup:cluster-category}
Let $\overline{\cat{C}_{n}}$ be a completed discrete cluster category of type $\BA$ with geometric model $(S,\overline{M_n})$.
Let $\cat{X}$ be a fan triangulation of $(S,\overline{M_n})$. 
\end{setup}

Recall that $\overline{\cat{C}_{n}}$ is a small triangulated category with split idempotents. 
And, by \cref{lem:precovering-triangulation-is-fan}, $\cat{X}$ is a contravariantly finite, additive subcategory of $\overline{\cat{C}_{n}}$. 

As before, $(\overline{\cat{C}_n},\BE,\fs)$ will denote the extriangulated category corresponding to the triangulated category $\overline{\cat{C}_n}$ with suspension $\sus$. 
The index $\indxx{\cat{X}}(-)=[-]_{\cat{X}}$ with respect to $\cat{X}$ takes values in $K_{0}(\overline{\cat{C}_n},\BE_{\cat{X}},\fs_{\cat{X}})$. 
In this section, we will exhibit several nice properties of this index. The first phenomenon in this case is that the index essentially takes values in $K^{\sp}_{0}(\cat{X})$, in which classically we would expect values of the index to lie.

\begin{thm}[=\cref{thm:B}]
\label{thm:fan-index-isom}
There is an isomorphism 
\begin{align*}
    K_0(\overline{\cat{C}_n},\BE_{\cat{X}},\fs_{\cat{X}}) 
    &\overset{\cong}{\longleftrightarrow}
    K^{\sp}_0(\cat{X})
    \\ 
    [C]_{\cat{X}}
    &\longmapsto [X_{0}]^{\sp} - [X_{1}]^{\sp}
    \\
    [X]_{\cat{X}}
    &\longmapsfrom [X]^{\sp},
\end{align*}
where 
$\begin{tikzcd}[column sep=0.5cm, cramped]
X_{1} 
    \arrow{r}
& X_{0} 
    \arrow{r}
& C 
    \arrow[dashed]{r}
& {}
\end{tikzcd}$
is an $\fs_{\cat{X}}$-triangle with $X_{i}\in\cat{X}$.
\end{thm}

\begin{notation}
Recall that $\indxx{\cat{X}}(C)$ denotes the class $[C]_{\cat{X}}$ in $K_0(\overline{\cat{C}_n},\BE_{\cat{X}},\fs_{\cat{X}})$.  
In light of \cref{thm:fan-index-isom}, in Section \ref{sec:index-in-cluster-cats} we will also denote the element $[X_{0}]^{\sp} - [X_{1}]^{\sp} \in K^{\sp}_0(\cat{X})$ by $\indxx{\cat{X}}(C)$ whenever  
$\begin{tikzcd}[column sep=0.5cm, cramped]
X_{1} 
    \arrow{r}
& X_{0} 
    \arrow{r}
& C 
    \arrow[dashed]{r}
& {}
\end{tikzcd}$
is an $\fs_{\cat{X}}$-triangle with $X_{i}\in\cat{X}$. 
Note that this assignment induces a group homomorphism 
$\indxx{\cat{X}}(-) \colon K^{\sp}_0(\ol{\cat{C}_n}) \to K^{\sp}_0(\cat{X})$.
\end{notation}

\begin{rem}
The Grothendieck groups of the completed and non-completed discrete cluster categories of type $\BA$ have been computed by Murphy in \cite[Thms.~1.1 and 1.2]{Murphy}.
\end{rem}

We need the next preparatory lemma to prove the above result.

\begin{lem}\label{lem:approximation}
For each object $C\in\overline{\cat{C}_n}$, there is a triangle 
$\begin{tikzcd}[column sep=0.5cm, cramped]
X_{1} 
    \arrow{r}{x}
& X_{0} 
    \arrow{r}{y}
& C 
    \arrow{r}{z}
& \sus X_{1}
\end{tikzcd}$
in $\overline{\cat{C}_n}$, where: 
\begin{enumerate}[label={(\roman*)}]
    \item $y$ is a (minimal) right $\cat{X}$-approximation of $C$; 
    \item $X_{1}\in\cat{X}$; and 
    \item $z\in\BE_{\cat{X}}(C,X_{1})$.
\end{enumerate}
\end{lem}

\begin{proof}
Let $C\in\overline{\cat{C}_n}$ be an arbitrary object. 
As $\cat{X}$ is contravariantly finite, we may take a right $\cat{X}$-approximation $y\colon X_{0} \to C$ of $C$ and complete it to the following triangle in $\ol{\cat{C}_n}$.
\begin{equation}\label{eqn:approx-triangle-for-C}
\begin{tikzcd}
A
    \arrow{r}{x}
& X_{0} 
    \arrow{r}{y}
& C 
    \arrow{r}{z}
& \sus A
\end{tikzcd}
\end{equation}
(Note that we can take $y$ to be right minimal by \cite[Cor.~1.4]{KrauseSaorin-minimal-approximations-of-modules} since $\ol{\cat{C}_n}$ is Krull-Schmidt, but this is not necessary for the arguments in the rest of this proof.) 
Since $y$ is a right $\cat{X}$-approximation, we have 
$\restr{\ol{\cat{C}_n}(-,z)}{\cat{X}} = 0$ and hence $z\in\BE_{\cat{X}}(C,A)$ by \eqref{eqn:relative-yoneda}. 

It suffices to show $A\in\cat{X}$ for then we may put $X_{1} \deff A$. 
Since $\cat{X}$ is a triangulation, it is enough to show $A$ does not cross $\cat{X}$ transversely. 
Thus, assume for contradiction that an indecomposable direct summand $B$ of $A$ does cross some arc $X\in\cat{X}$ in this way.
This means $\ol{\cat{C}_n}(X,\sus B) \neq 0$ (see \ref{homspaces} from \cref{sec:completed-cluster-cat}), and so let $b\colon X \to \sus B$ be a non-zero morphism. 
Post-composing with the canonical inclusion $\iota \colon \sus B \into \sus A$, we obtain a non-zero morphism $a \deff \iota b\in\ol{\cat{C}_n}(X,\sus A)$. 
Since \eqref{eqn:approx-triangle-for-C} is a triangle in $\ol{\cat{C}_n}$, we have the following exact sequence in which the rightmost homomorphism is injective.
\[
\begin{tikzcd}[column sep=2cm]
\ol{\cat{C}_n}(X,X_{0})
    \arrow{r}{\ol{\cat{C}_n}(X,y)}
& \ol{\cat{C}_n}(X,C)
    \arrow{r}{\ol{\cat{C}_n}(X,z) = 0}
&[0.2cm] \ol{\cat{C}_n}(X,\sus A)
    \arrow[hook]{r}{\ol{\cat{C}_n}(X,\sus x)}
& \ol{\cat{C}_n}(X,\sus X_{0})
\end{tikzcd}
\]
In particular, the composition $(\sus x) a \colon X \to \sus X_{0}$ is non-zero. Then there is an indecomposable direct summand $\sus Y$ of $\sus X_{0}$, with canonical retract $\pi \colon \sus X_{0} \onto \sus Y$, such that
$\pi (\sus x) a 
    \colon X \to \sus Y
$ 
is non-zero. 
This gives the following commutative diagram.
\[
\begin{tikzcd}
X 
    \arrow{rr}{\pi (\sus x)a \neq 0}
    \arrow{dr}[swap]{b}
&& \sus Y 
\\
& \sus B 
    \arrow{ur}[swap]{\pi(\sus x)\iota}
&
\end{tikzcd}
\]

Since both $X$ and $Y$ are part of the triangulation $\cat{X}$ and 
$\ol{\cat{C}_n}(X,\sus Y) \neq 0$, by \ref{homspaces} we must have either:
\begin{itemize}
    \item $\{y_0, p\} = Y \neq X = \{x_0, p\}$ but share exactly one accumulation point $p$ as an endpoint, and we can move $x_0$ anticlockwise along $\partial S$ to reach $y_0$ without passing through $p$; or
    \item $Y = X = \{x_0, x_1\}$ has both endpoints as accumulation points.
\end{itemize}
In the first case, using \ref{morph_from_C} and that $\pi(\sus x)a \colon X \to \sus Y$ factors through $\sus B$, we must have that $\sus B$ also has one endpoint at the accumulation point $p$ and its other endpoint lies in $[x_0, y^{-}_0]$. That is, $B = \{b_0,p\}$ with $b_0 \in [x^{+}_{0},y_{0}]$. But then $B$ does not cross $X$ transversely, which contradicts our assumption. 
Hence, we are in the second case and it follows that 
$B = \sus B = Y = X \in \cat{X}$. But then $B$ and $X$ do not cross transversely, a contradiction. 
Therefore, no direct summand of $A$ crosses $\cat{X}$ transversely, so by maximality of the triangulation $\cat{X}$ we must have $A\in\cat{X}$ and we are done.
\end{proof}

We are ready to prove \cref{thm:fan-index-isom}.

\begin{proof}[Proof of \cref{thm:fan-index-isom}]
We wish to apply \cite[Thm.~4.5]{Ogawa-Shah-resolution} to the small extriangulated category $(\ol{\cat{C}_n},\BE_{\cat{X}},\fs_{\cat{X}})$ and the additive subcategory $\cat{X}\sse\ol{\cat{C}_n}$. 

First, we note that $\cat{X}$ is precisely the subcategory of $\BE_{\cat{X}}$-projectives 
in $(\ol{\cat{C}_n},\BE_{\cat{X}},\fs_{\cat{X}})$ 
(see \cite[Def.~3.23]{NakaokaPalu-extriangulated-categories-hovey-twin-cotorsion-pairs-and-model-structures}), and hence $\cat{X}$ is extension-closed in $(\ol{\cat{C}_n},\BE_{\cat{X}},\fs_{\cat{X}})$. 
Thus, by \cite[Rem.~2.18]{NakaokaPalu-extriangulated-categories-hovey-twin-cotorsion-pairs-and-model-structures}, there is an extriangulated category 
$(\cat{X},\restr{\BE_{\cat{X}}}{\cat{X}},\restr{\fs_{\cat{X}}}{\cat{X}})$ arising from 
$\cat{X}$ inheriting an extriangulated structure from $(\ol{\cat{C}_n},\BE_{\cat{X}},\fs_{\cat{X}})$. 
But this structure just corresponds to the split exact structure on $\cat{X}$ as $\restr{\BE_{\cat{X}}}{\cat{X}} = 0$. In particular, 
$K_{0}(\cat{X},\restr{\BE_{\cat{X}}}{\cat{X}},\restr{\fs_{\cat{X}}}{\cat{X}}) = K^{\sp}_{0}(\cat{X})$ is just the split Grothendieck group of $\cat{X}$.

Second, if $y\colon X \to Y$ is an $\fs_{\cat{X}}$-deflation with $X,Y\in\cat{X}$, then it is a split epimorphism as $Y$ is $\BE_{\cat{X}}$-projective. Moreover, the cocone of $y$ is thus a direct summand of $X$ and hence also lies in $\cat{X}$. That is, $\cat{X}$ is closed under cocones of $\fs_{\cat{X}}$-deflations.

\cref{lem:approximation} tells us that each object $C\in\ol{\cat{C}_n}$ admits an $\cat{X}$-resolution of length $1$ in $(\ol{\cat{C}_n},\BE_{\cat{X}},\fs_{\cat{X}})$ (see \cite[Def.~4.2]{Ogawa-Shah-resolution}). 
Hence, the hypotheses of \cite[Thm.~4.5]{Ogawa-Shah-resolution} are met and the claim follows.
\end{proof}

\begin{rem}
Given that a fan triangulation $\cat{X}$ is functorially finite in $\ol{\cat{C}_n}$ (see \cref{lem:precovering-triangulation-is-fan}) and becomes \emph{rigid} in the extriangulated category $(\ol{\cat{C}_n},\BE_{\cat{X}},\fs_{\cat{X}})$ (i.e.\ $\BE_{\cat{X}}(\cat{X},\cat{X}) = 0$), it is natural to ask if $\cat{X}$ is cluster tilting in this extriangulated category (see \cite[Def.~5.3]{LiuNakaoka}).
Although we always have
$
\cat{X}
     = \set{A\in\ol{\cat{C}_n} | \BE_{\cat{X}}(A,\cat{X}) = 0}
$
using \cref{lem:approximation}, 
the set 
$
\set{A\in\ol{\cat{C}_n} | \BE_{\cat{X}}(\cat{X},A) = 0}
$
will never be equal to $\cat{X}$. Indeed, since $\cat{X}$ is the subcategory of $\BE_{\cat{X}}$-projectives, this latter set is just all of $\ol{\cat{C}_n}$. 
\end{rem}

%%%%%%%%%%%%%%%%%%%%%%%%%%%%%%%%%%%%%%%%%%%%%%%%%%%%%%%%%%%%%
\subsection{The index determines a rigid object}
\label{sec:index-determines-rigids}

In the remainder \cref{sec:index-in-cluster-cats}, we establish results inspired by Dehy--Keller \cite[Sec.~2]{DehyKeller-On-the-combinatorics-of-rigid-objects-in-2-Calabi-Yau-categories}. 
We still work under \cref{setup:cluster-category}.

\begin{lem}
\label{lem:no-common-summands}
Suppose $C\in\ol{\cat{C}_n}$ is a rigid object and let  
$\begin{tikzcd}[column sep=0.5cm, cramped]
X_{1} 
    \arrow{r}{x}
& X_{0} 
    \arrow{r}{y}
& C 
    \arrow{r}{z}
& \sus X_{1}
\end{tikzcd}$
be a triangle in $\overline{\cat{C}_n}$, where $y$ is a minimal $\cat{X}$-approximation of $C$. 
Then, even up to isomorphism, $X_0$ and $X_1$ have no direct summands in common.
\end{lem}

\begin{proof}
First note that since $y$ is right minimal in the Krull-Schmidt category $\ol{\cat{C}_n}$, we see that $x \colon X_1 \to X_0$ is radical. 
Now suppose $f\colon X_1 \to X_0$ is any morphism in $\ol{\cat{C}_n}$. 
Since $C$ is rigid, the morphism $-y f (\sus^{-1}z) \colon \sus^{-1}C \to C$ vanishes, which yields a morphism $e\colon \sus^{-1}C \to X_1$ so that $-f\sus^{-1}z = xe$. The axioms of a triangulated category yield the following morphism of triangles.
\begin{equation}\label{eqn:radical-triangles}
\begin{tikzcd}[column sep=1.5cm]
\sus^{-1}C 
    \arrow{r}{-\sus^{-1}z}
    \arrow{d}{e}
& X_1 
    \arrow{r}{x}
    \arrow{d}{f}
    \arrow[dotted]{dl}[swap]{h_1}
& X_0 
    \arrow{r}{y}
    \arrow{d}{g}
    \arrow[dotted]{dl}[swap]{h_0}
& C 
    \arrow{d}{\sus e}
\\
X_1 
    \arrow{r}{x}
& X_0 
    \arrow{r}{y}
& C
    \arrow{r}{z}
& \sus X_1
\end{tikzcd}
\end{equation}
Since $y$ is a right $\cat{X}$-approximation, there exists $h_0\colon X_0 \to X_0$ such that $g = yh_0$. 
Then $y(h_0 x - f) = 0$ implies the existence of a morphism $h_1\colon X_1 \to X_1$ with $f = h_0 x - x h_1$ (see \eqref{eqn:radical-triangles}). 
Since $x$ is radical, we see that $f$ must be radical too. The assertion follows as each component of $f$ between indecomposable summands must be radical by \cite[Prop.~1.1(a)]{Bautista} and a non-isomorphism by \cite[Prop.~2.1(b)]{Bautista}.
\end{proof}

In light of \cref{lem:no-common-summands}, one may follow the proof of \cite[Thm.~2.3]{DehyKeller-On-the-combinatorics-of-rigid-objects-in-2-Calabi-Yau-categories} to see that a rigid object $C$ in $\ol{\cat{C}_n}$ is determined by its index up to isomorphism.

\begin{thm}[=\cref{thm:C}]
\label{thm:index_det_rigid}
There is an injection from the set of isomorphism classes of rigid objects in $\ol{\cat{C}_n}$ to $K^{\sp}_{0}(\cat{X})$ given by 
$[C] \mapsto \indxx{\cat{X}}(C)$.
\end{thm}

%%%%%%%%%%%%%%%%%%%%%%%%%%%%%%%%%%%%%%%%%%%%%%%%%%%%%%%%%%%%%
\subsection{Sign-coherence, linear independence and bases}
\label{sec:signs-independence-basis}

We still work under \cref{setup:cluster-category}. 
In the rest, we fix the canonical basis of $K^{\sp}_{0}(\cat{X})$ given by the isomorphism classes of the indecomposables in $\cat{X}$. The next result means that the indices of the direct summands of a rigid object all lie in the same hyperquadrant of $K^{\sp}_{0}(\cat{X})$. The proof is the same as in \cite[Sec.~2.4]{DehyKeller-On-the-combinatorics-of-rigid-objects-in-2-Calabi-Yau-categories}; the definition of \emph{sign-coherent} can be found therein.

\begin{thm}[=\cref{thm:D}]
\label{prop:rigid}
Suppose $C\in\ol{\cat{C}_n}$ is rigid and that $U,V$ are direct summands of $C$. Then $\{ \indxx{\cat{X}}(U), \indxx{\cat{X}}(V) \}$ is a sign-coherent subset of $K^{\sp}_{0}(\cat{X})$. 
\end{thm}

Following the argument in \cite[Sec.~2.5]{DehyKeller-On-the-combinatorics-of-rigid-objects-in-2-Calabi-Yau-categories}, we observe that the indices of direct summands of a rigid object are linearly independent.

\begin{thm}[=\cref{thm:E}]
\label{prop:linearly-independent}
Suppose $C$ is a rigid object in $\ol{\cat{C}_n}$ and let $\cat{U}$ be a finite set of indecomposable, pairwise non-isomorphic direct summands of $C$. 
Then the set 
$
\set{ \indxx{\cat{X}}(U) | U\in\cat{U} }
$
is linearly independent in $K^{\sp}_{0}(\cat{X})$.
\end{thm}

%%%%%%%%%%%%%%%%%%%%%%%%%%%%%%%%%%%%%%%%%%%%%%%%%%%%%%%%%%%%%
\subsection{Rigid fan triangulations and mutation}
\label{sec:mutation}

We still work under \cref{setup:cluster-category}. 
For the analogue of \cite[Thm.~2.4]{DehyKeller-On-the-combinatorics-of-rigid-objects-in-2-Calabi-Yau-categories}, we need to assume the fan triangulations in consideration are also rigid (see \cref{exam:2-accum-pt}). Although this is not true in general (see \cref{exam:1-accum-pt,exam:2-accum-pt}), such examples do exist 
(see e.g.\ \cite[Fig.~7]{PaquetteYildirim}).

\begin{thm}
Suppose $\cat{X},\cat{Y}$ are rigid, fan triangulations of $(S,\ol{M_n})$. 
Then $\cat{X},\cat{Y}$ are cluster tilting in the triangulated category $\ol{\cat{C}_n}$ and 
the set 
$
\set{ \indxx{\cat{X}}(Y) | Y\in\cat{Y} }
$
is also a basis of $K^{\sp}_{0}(\cat{X})$.
\end{thm}

\begin{proof}
Note that any rigid, fan triangulation is cluster tilting by \cref{prop_ct_PY}. 
Consequently, note that $\indxx{\cat{X}}(-)$ is the same index considered in \cite{DehyKeller-On-the-combinatorics-of-rigid-objects-in-2-Calabi-Yau-categories}. Thus, the last assertion follows from \cite[Thm.~2.4]{DehyKeller-On-the-combinatorics-of-rigid-objects-in-2-Calabi-Yau-categories}.
\end{proof}

In line with \cite[Def.~5.1]{CKP}, 
we say that a triangulation $\cat{Y}$ of $(S,\ol{M_n})$ is \emph{the mutation of $\cat{X}$ at an arc $X\in\cat{X}$} if there is a unique arc $Y\in\cat{Y}$ such that $X\neq Y$ and $\cat{Y} = (\cat{X}\setminus \{X\})\cup \{Y\}$. 
Note that this definition only requires $\cat{X}$ to be a triangulation of $(S,\ol{M_n})$. 
By \cite[Thm.~3]{CKP}, the arcs $X$ and $Y$ must be the diagonal of a quadrilateral with edges in $\cat{X} \cap \cat{Y}$ or homotopic to an unmarked boundary segment, as indicated in Figure~\ref{fig:mutation}.
\begin{figure}[ht]
\centering
\begin{tikzpicture}[scale=2]
\begin{scope}
    \draw (0,0) circle (1cm); 
    
    %points
    \draw (45:0.97cm) edge node[above right, pos=0.3] {$\scriptstyle y_0$} (45:1.03cm);
    \draw (135:0.97cm) edge node[above left, pos=0.3] {$\scriptstyle x_1$} (135:1.03cm);
    \draw (225:0.97cm) edge node[below left, pos=0.3] {$\scriptstyle y_1$} (225:1.03cm);
    \draw (315:0.97cm) edge node[below right, pos=0.3] {$\scriptstyle x_0$} (315:1.03cm);

    % arcs
    \draw (315:1cm) edge node[right, pos=0.7] {$\scriptstyle X$}(135:1cm); 
    \draw[red] (225:1cm) edge node[left, pos=0.7, red] {$\scriptstyle Y$} (45:1cm); 
    
    \draw (135:1cm) edge node[right, pos=0.5] {$\scriptstyle T_1$}(225:1cm); 
    \draw (45:1cm) edge node[left, pos=0.5] {$\scriptstyle T_2$}(315:1cm); 
    \draw (45:1cm) edge node[below, pos=0.5] {$\scriptstyle S_1$}(135:1cm); 
    \draw (315:1cm) edge node[above, pos=0.5] {$\scriptstyle S_2$}(225:1cm); 
\end{scope}
\end{tikzpicture}
\caption{The arc $X\in\cat{X}$ and its flip $Y\in\cat{Y}\setminus\cat{X}$ lying as diagonals in a quadrilateral with edges in $\cat{X}\cap \cat{Y}$ or homotopic to an unmarked boundary segment. Note that the marked points $x_0,x_1,y_0,y_1$ must all be distinct.} 
    \label{fig:mutation}
\end{figure}
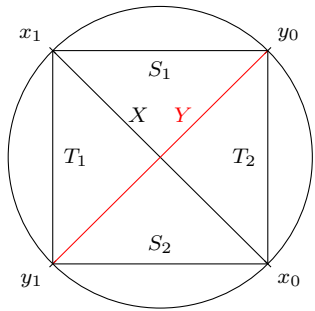
Moreover, by \ref{triangles-transverse}, there are the following exchange triangles.
\begin{equation}\label{eqn:exchange-triangles}
\begin{tikzcd}[column sep=0.5cm]
X
    \arrow{r}
& S_1 \oplus S_2
    \arrow{r}
& Y
    \arrow{r}
& \sus X
\end{tikzcd}
\hspace{0.5cm}
\text{and}
\hspace{0.5cm}
\begin{tikzcd}[column sep=0.5cm]
Y
    \arrow{r}
& T_1 \oplus T_2
    \arrow{r}
& X
    \arrow{r}
& \sus Y
\end{tikzcd}
\end{equation}
See also \cite[Rem.~5.4]{CKP}.

\begin{lem}\label{lem:mutation-of-rigid-fan-triangulation}
Suppose $\cat{Y}$ is the mutation of a triangulation $\cat{X}$. 
If $\cat{X}$ is rigid (resp.\ a fan triangulation), then so too is $\cat{Y}$. 
\end{lem}

\begin{proof}
Suppose $\cat{Y} = (\cat{X}\setminus \{X\})\cup \{Y\}$ is the mutation of $\cat{X}$ at an arc $X\in\cat{X}$ as depicted in Figure~\ref{fig:mutation}.

If $\cat{X}$ is rigid, then none of the distinct points $x_0,x_1,y_0,y_1$ can be accumulation points. Indeed, if $p \in \set{x_0,x_1,y_0,y_1}$ were an accumulation point, then there would be an extension between the arcs $S_i$ and $T_j$ $(i,j\in\{1,2\})$ that were incident with $p$. 
And, flipping $X$ to $Y$ does not introduce any transverse crossings, so $\cat{Y}$ must also be rigid.

Suppose now that $\cat{X}$ is a fan triangulation and assume for a contradiction that $\cat{Y}$ is not, that is, $\cat{Y}$ contains an infinite leapfrog $\cat{L}$. By Definition~\ref{def:fan}, there is a curve $\gamma$ in $S$ crossing all the infinitely many arcs in $\cat{L}$. Since $\cat{Y}$ and $\cat{X}$ only differ by one arc, such $\gamma$ also crosses an infinite set of consecutive arcs in $\cat{L}\cap\cat{X}$, giving an infinite leapfrog in $\cat{X}$. This is a contradiction to $\cat{X}$ being a fan triangulation. 
\end{proof}

By \cref{lem:mutation-of-rigid-fan-triangulation}, we know that the mutation $\cat{Y}$ of a (rigid) fan triangulation $\cat{X}$ is again a (rigid) fan triangulation. In particular, we may consider the two indices 
$\indxx{\cat{X}}(-)$
and 
$\indxx{\cat{Y}}(-)$ 
with respect to $\cat{X}$ and $\cat{Y}$, respectively. 
If \eqref{eqn:exchange-triangles} are the corresponding exchange triangles, then 
we need two group homomorphisms 
\mbox{$\phi,\psi\colon K^{\sp}_0(\cat{X}) \to K^{\sp}_0(\cat{Y})$} to see how the index behaves under mutation (cf.\ \cite[Sec.~3]{DehyKeller-On-the-combinatorics-of-rigid-objects-in-2-Calabi-Yau-categories}). 
On a generator $[W]^{\sp} \in  K^{\sp}_0(\cat{X})$, we define: 
\[
\phi([W]^{\sp}) \deff
	\begin{cases}
		[W]^{\sp} & \text{if $W \not\cong X$}\\
		[T_{1}]^{\sp} + [T_{2}]^{\sp} - [Y]^{\sp} & \text{if $W = X$,}
	\end{cases}
\]
and 
\[
\psi([W]^{\sp}) \deff
	\begin{cases}
		[W]^{\sp} & \text{if $W \neq X$}\\
		[S_{1}]^{\sp} + [S_{2}]^{\sp} - [Y]^{\sp} & \text{if $W = X$.}
	\end{cases}
\]
For $C\in\ol{\cat{C}_n}$, we denote by 
$[ \indxx{\cat{X}}(C) \mathbin{:} X]$ the coefficient of $[X]^{\sp}$ in $\indxx{\cat{X}}(C)$ with respect to the basis of $K^{\sp}_{0}(\cat{X})$ given by the isomorphism classes of the indecomposables in $\cat{X}$. 

The last main result of this section is as follows, and the proof of \cite[Thm.~3.1]{DehyKeller-On-the-combinatorics-of-rigid-objects-in-2-Calabi-Yau-categories} carries over thanks to \cref{lem:no-common-summands}.

\begin{thm}\label{thm:mutation}
Suppose $\cat{X},\cat{Y}$ are rigid, fan triangulations of $(S,\ol{M_n})$ that are related by one flip with exchange triangles \eqref{eqn:exchange-triangles}. 
For a rigid object $C\in\ol{\cat{C}_n}$, we have 
\[
\indxx{\cat{Y}}(C) = 
	\begin{cases}
		\phi(\indxx{\cat{X}}(C)) & \text{if $[ \indxx{\cat{X}}(C) \mathbin{:} X] \geq 0$}\\
		\psi(\indxx{\cat{X}}(C)) & \text{if $[ \indxx{\cat{X}}(C) \mathbin{:} X] \leq 0$.}
	\end{cases}
\]
\end{thm}

%%%%%%%%%%%%%%%%%%%%%%%
\subsection{Examples}
\label{sec:examples}

We now give some examples to illustrate the theory above.

\begin{example}\label{exam:1-accum-pt}
Consider the completed discrete cluster category $\ol{\cat{C}_1}$ of type $\BA$ associated to the disc with one accumulation point $p$ on the boundary. 
Let $\cat{X}$ denote the triangulation consisting of all arcs having precisely one endpoint at $p$; see Figure~\ref{fig:1-acc-pt}. 
\begin{figure}[ht]
\centering
\begin{tikzpicture}[scale=2]
\begin{scope}
    \draw (0,0) circle (1cm); 

    % accum pt
    \draw (90:1cm) node{$\bullet$}; 
    \draw (90:1.18cm) node{$\scriptstyle p$};

    % arcs
    \draw (90:1cm) -- (270:1cm); % straight down

  %left
    %\draw (90:1cm) -- (250:1cm);
    %\draw (90:1cm) -- (230:1cm);
    %\draw (90:1cm) -- (210:1cm);
    %\draw (90:1cm) -- (190:1cm);
    %\draw (90:1cm) -- (170:1cm);
%bent version left
\draw (90:1cm) edge[bend left] (150:1cm);
\draw (90:1cm) edge[bend left=20] (170:1cm);
\draw (90:1cm) edge[bend left=15] (190:1cm);
\draw (90:1cm) edge[bend left=10] (210:1cm);
\draw (90:1cm) edge[bend left=10] (230:1cm);
\draw (90:1cm) edge[bend left=5] (250:1cm);

    %right
    %\draw (90:1cm) -- (290:1cm);
    %\draw (90:1cm) -- (310:1cm);
    %\draw (90:1cm) -- (330:1cm);
    %\draw (90:1cm) -- (350:1cm);
    %\draw (90:1cm) -- (370:1cm);
%bent version right
\draw (90:1cm) edge[bend right] (390:1cm);
\draw (90:1cm) edge[bend right=20] (370:1cm);
\draw (90:1cm) edge[bend right=15] (350:1cm);
\draw (90:1cm) edge[bend right=10] (330:1cm);
\draw (90:1cm) edge[bend right=10] (310:1cm);
\draw (90:1cm) edge[bend right=5] (290:1cm);

    %dots
   \draw[thick, dotted] ([shift=(130:0.83cm)]0,0) arc (-100:-140:0.2cm);
   \draw[thick, dotted] ([shift=(50:0.83cm)]0,0) arc (-80:-40:0.2cm);
\end{scope}
\end{tikzpicture}
\caption{The fountain triangulation $\cat{X}$ in $(S,\ol{M_1})$ with base the unique accumulation point $p$.} 
    \label{fig:1-acc-pt}
\end{figure}

The triangulation $\cat{X}$ is known as the \emph{fountain triangulation with base $p$}, and is thus a fan triangulation \cite[Exam.~5.5]{CKP}. 
Moreover, $\cat{X}$ is an example of a fan triangulation that is neither rigid nor closed under extensions in the triangulated category $\ol{\cat{C}_1}$. Indeed, fix any indecomposable $X = \{p,x\}\in\cat{X}$, and consider 
$\sus^{-1}X = \{p, x^+\}$
and 
$\sus^{-2}X = \{p, x^{++}\}$ which also lie in $\cat{X}$.
Then there is a non-zero morphism $X \to \sus^{-1}X = \sus(\sus^{-2}X)$, which completes backwards to the triangle
$
\sus^{-2}X \to E \to X \to \sus^{-1}X,
$
where $E = \{ x, x^{++} \}$ (see \ref{triangle-one-acc-pt}). 
Notice that $E\notin\cat{X}$ even though $\sus^{-2}X,X\in\cat{X}$.
\end{example}

\begin{example}\label{exam:2-accum-pt}
Consider the completed discrete cluster category $\ol{\cat{C}_2}$ of type $\BA$ associated to the disc with two accumulation points $p \neq q$ on the boundary. Let $\cat{X}$ denote the fountain triangulation with base $p$; 
see Figure~\ref{fig:2-acc-pt}. 
\begin{figure}[ht]
\centering
\begin{tikzpicture}[scale=2]
\begin{scope}
    \draw (0,0) circle (1cm); 

    % accum pts
    \draw (90:1cm) node{$\bullet$}; 
    \draw (90:1.18cm) node{$\scriptstyle p$};
    \draw (310:1cm) node{$\bullet$}; 
    \draw (310:1.18cm) node{$\scriptstyle q$};

    % arcs
    \draw (90:1cm) -- (270:1cm); % straight down

%bent version left
\draw (90:1cm) edge[bend left] (150:1cm);
\draw (90:1cm) edge[bend left=20] (170:1cm);
\draw (90:1cm) edge[bend left=15] (190:1cm);
\draw (90:1cm) edge[bend left=10] (210:1cm);
\draw (90:1cm) edge[bend left=10] (230:1cm);
\draw (90:1cm) edge[bend left=5] (250:1cm);

%bent version right
\draw (90:1cm) edge[bend right] (390:1cm);
\draw (90:1cm) edge[bend right=20] (370:1cm);
\draw (90:1cm) edge[bend right=15] (350:1cm);
\draw (90:1cm) edge[bend right=10] (335:1cm);
\draw (90:1cm) edge[bend right=10] (325:1cm);
\draw (90:1cm) edge[bend right=10] (310:1cm);
\draw (90:1cm) edge[bend right=5] (295:1cm);
\draw (90:1cm) edge[bend right=5] (285:1cm);

    %dots
   \draw[thick, dotted] ([shift=(130:0.83cm)]0,0) arc (-100:-140:0.2cm);
   \draw[thick, dotted] ([shift=(50:0.83cm)]0,0) arc (-80:-40:0.2cm);
   \draw[thick, dotted] ([shift=(-44:0.83cm)]0,0) arc (-70:-30:0.2cm);
   \draw[thick, dotted] ([shift=(-60:0.8cm)]0,0) arc (-80:-40:0.2cm);
\end{scope}
\end{tikzpicture}
\caption{The fountain triangulation $\cat{X}$ in $(S,\ol{M_2})$ with base the accumulation point $p$.} 
    \label{fig:2-acc-pt}
\end{figure}
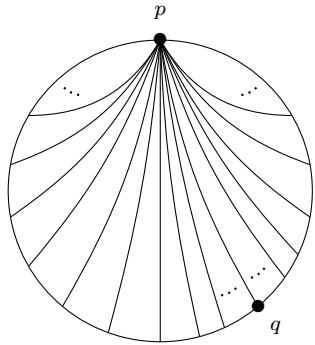

\begin{itemize}

\item We now show that in \cref{thm:index_det_rigid} the assumption of $C$ being rigid  is necessary and, in particular, assuming only that the indecomposable direct summands of $C$ are pairwise non-transversely crossing is not enough.
Consider Figure~\ref{fig:2-acc-pt_extriangles} for the rest of the argument. 
\begin{figure}[ht]
\centering
\begin{tikzpicture}[scale=2]
    \draw (0,0) circle (1cm); 

    % accum pts
    \draw (90:1cm) node{$\bullet$}; 
    \draw (90:1.18cm) node{$\scriptstyle p$};
    \draw (310:1cm) node{$\bullet$}; 
    \draw (310:1.18cm) node{$\scriptstyle q$};

    % arcs
    \draw (310:1cm) edge node[above, pos=0.3] {$\scriptstyle A$}(240:1cm); 
    \draw (390:1cm) edge node[left, pos=0.5] {$\scriptstyle E$} (240:1cm); 
    \draw (390:1cm) edge node[below left, pos=0.4] {$\scriptstyle C$} (310:1cm);

    %points
    \draw (240:0.97cm) edge node[below left, pos=0.3] {$\scriptstyle a_1$} (240:1.03cm);
    \draw (250:0.97cm) edge node[below, pos=0.3] {$\scriptstyle a_1^+$} (250:1.03cm);
    \draw (390:0.97cm) -- (390:1.03cm);

    %covers
    \draw (90:1cm) edge[red] node[below, pos=0.4] {$\scriptstyle X_C$} (390:1cm);
    \draw (90:1cm) edge[red] node[left, pos=0.5] {$\scriptstyle X_A$} (310:1cm);
    \draw (90:1cm) edge[red] node[left, pos=0.3] {$\scriptstyle Y$} (250:1cm);
\end{tikzpicture}
\caption{The black arcs 
give the $\fs_{\cat{X}}$-triangle 
$A \to E \to C \dashrightarrow$.}
    \label{fig:2-acc-pt_extriangles}
\end{figure}
The object $A\oplus C$ is not rigid, as by \ref{homspaces} there is a non-zero morphism $C\rightarrow \sus A$. Note that $E$ is rigid.

By \ref{morph_to_C} and \cref{lem:approximation} we have $\fs_{\cat{X}}$-triangles
\begin{center}
$
\begin{tikzcd}[column sep=0.5cm,cramped]
X_A \arrow{r}
& X_C \arrow{r}{x_C}
& C \arrow[dashed]{r}{c}
& {}, \quad
Y \arrow{r}
& X_A \arrow{r}{x_A}
& A \arrow[dashed]{r}{a}
& {}, \quad
Y \arrow{r}
& X_C \arrow{r}{x_E}
& E \arrow[dashed]{r}{e}
& {},
\end{tikzcd}
$
\end{center}
where $x_C$, $x_A$ and $x_E$ are minimal right $\cat{X}$-approximations and $X_A,\, Y\in\cat{X}$. Moreover, noting that the direct sum of $x_A$ and $x_C$ is a minimal right $\cat{X}$-ap\-prox\-i\-ma\-tion of $A\oplus C$, by \cref{thm:fan-index-isom} we have 
\begin{center}
$
    \indxx{\cat{X}}(A\oplus C)= [X_A\oplus X_C]^{\sp}-[Y\oplus X_A]^{\sp}=[X_C]^{\sp}-[Y]^{\sp}=\indxx{\cat{X}}(E).
$
\end{center}
But $A\oplus C \not\cong E$, hence the map from \cref{thm:index_det_rigid} is an injection only when we restrict to rigid objects.
\item Finally, we show that \cite[Thm.~2.4]{DehyKeller-On-the-combinatorics-of-rigid-objects-in-2-Calabi-Yau-categories} fails for fan triangulations that are not rigid. By the above point, for the arcs in Figure~\ref{fig:2-acc-pt_extriangles}, we have
\begin{center}
$
    \indxx{\cat{X}}(A)+ \indxx{\cat{X}}(C)=\indxx{\cat{X}}(E).
$
\end{center}
Moreover, the set of arcs $A$, $C$ and $E$ can be extended to a fan triangulation; see, for example, the fan triangulation $\cat{X}'$ in Figure~\ref{fig:2-acc-pt_triangultation2}. Then, for $X'\deff A\oplus C\oplus E\in\cat{X}'$, the set $\{\indxx{\cat{X}}(A),\indxx{\cat{X}}(C),\indxx{\cat{X}}(E)\}$ is not linearly independent in  $K^{\sp}_{0}(\cat{X})$.
\end{itemize}
\end{example}
\begin{figure}[ht]
\centering
\begin{tikzpicture}[scale=2]
\begin{scope}
    \draw (0,0) circle (1cm); 

    % accum pts
    \draw (90:1cm) node{$\bullet$}; 
    \draw (90:1.18cm) node{$\scriptstyle p$};
    \draw (310:1cm) node{$\bullet$}; 
    \draw (310:1.18cm) node{$\scriptstyle q$};

    % arcs A C E
    \draw (310:1cm) edge[bend right=20] node[above, pos=0.3] {$\scriptstyle A$}(240:1cm); 
    \draw (390:1cm) edge node[right, pos=0.5] {$\scriptstyle E$} (240:1cm); 
    \draw (390:1cm) edge[bend right=10] node[below left, pos=0.4] {$\scriptstyle C$} (310:1cm);

    %points
    \draw (240:0.97cm) edge (240:1.03cm);
    \draw (250:0.97cm) -- (250:1.03cm);
    \draw (260:0.97cm) --  (260:1.03cm);
    \draw (390:0.97cm) -- (390:1.03cm);
    \draw (230:0.97cm) -- (230:1.03cm);
    \draw (220:0.97cm) -- (220:1.03cm);
    \draw (10:0.97cm) -- (10:1.03cm);
    \draw (20:0.97cm) -- (20:1.03cm);
    \draw (40:0.97cm) -- (40:1.03cm);
    \draw (50:0.97cm) -- (50:1.03cm);

    %arcs triangulation
    \draw (240:1cm) edge[bend left=40] (260:1cm);
    \draw (240:1cm) edge[bend left=30] (275:1cm);
    \draw (240:1cm) edge[bend left=30] (290:1cm);
    \draw (30:1cm) edge[bend right=5] (230:1cm);
    \draw (30:1cm) edge[bend right=5] (220:1cm);
    \draw (30:1cm) edge[bend right=5] (205:1cm);
    \draw (30:1cm) edge[bend right=5] (185:1cm);
    \draw (30:1cm) (165:1cm);
    \draw (30:1cm) edge[bend left=5] (155:1cm);
    \draw (30:1cm) edge[bend left=10] (135:1cm);
    \draw (30:1cm) edge[bend left=10] (115:1cm);
    \draw (30:1cm) edge[bend left=20] (90:1cm);
    \draw (30:1cm) edge[bend left=20] (70:1cm);
    \draw (30:1cm) edge[bend left=30] (50:1cm);
    \draw (30:1cm) edge[bend right=30] (10:1cm);
    \draw (30:1cm) edge[bend right=20] (-10:1cm);
    \draw (30:1cm) edge[bend right=10] (-30:1cm);
    
    %dots
    \draw[thick, dotted] ([shift=(290:0.90cm)]0,0) arc (-70:-30:0.2cm);
    \draw[thick, dotted] ([shift=(320:0.90cm)]0,0) arc (-60:-20:0.2cm);
    \draw[thick, dotted] ([shift=(100:0.88cm)]0,0) arc (130:90:0.2cm);
    \draw[thick, dotted] ([shift=(80:0.93cm)]0,0) arc (100:60:0.2cm);
\end{scope}
\end{tikzpicture}
\caption{A fan triangulation $\cat{X}'$ including the arcs $A$, $C$ and $E$.} 
    \label{fig:2-acc-pt_triangultation2}
\end{figure}
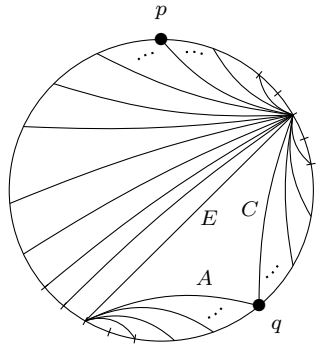

%%%%%%%%%%%%%%%%%%%%%%%%%%%%%%%%%%%%%%%%%%%%%%%%
%%%%%%%%%%%%%%%%%%%%%%%%%%%%%%%%%%%%%%%%%%%%%%%%

\section{The module category of an additive category}
\label{sec:module-category}

In the remaining two sections of this paper, our aim is to establish a formula for the additivity of the index on triangles up to an error term. 
We work slightly more generally in this section, assuming only that we have a nice enough additive category $\cat{C}$ with a contravariantly finite, additive subcategory $\cat{X}$ (see \cref{setup_additive}). 
We construct an epimorphism $K_0(\rmod{\cat{C}})\onto K_0(\rmod{\cat{X}})$ of Grothendieck groups and give a description of its kernel (\cref{prop:restriction-right-exact-sequence-of-K0}), which is key in the proof of \cref{thm:F}.

\begin{notation}\label{notation_additive}
For the following, suppose that $\cat{C}$ is a skeletally small additive category.
\begin{enumerate}[label=\textup{(\arabic*)}]

    \item Let $\cat{X}\subseteq \cat{C}$ be a subcategory and $C$ an object in $\cat{C}$. A morphism $x\colon  X\rightarrow C$ with $X\in\cat{X}$ is called a \textit{right $\cat{X}$-approximation of $C$} if for any $X'\in\cat{X}$ the induced morphism $\cat{C}(X',x)\colon \cat{C}(X',X)\rightarrow\cat{C}(X',C)$ is surjective. If for each $C\in\cat{C}$ there is a right $\cat{X}$-approximation, then we say that $\cat{X}$ is \textit{contravariantly finite} (see \cite[pp.~113--114]{auslander-reiten-contravariantly}).

    \item A morphism $a\colon  A\rightarrow B$ in $\cat{C}$ is called a \textit{weak kernel of a morphism $b\colon  B\rightarrow C$ in $\cat{C}$} if the induced sequence
    \[
    \begin{tikzcd}[column sep=1.5cm]
        \cat{C}(-,A)
            \arrow{r}{\cat{C}(-,a)}
        & \cat{C}(-,B)
            \arrow{r}{\cat{C}(-,b)}
        & \cat{C}(-,C)
    \end{tikzcd}
    \]
    is exact. We say that the category $\cat{C}$ \textit{has weak kernels} if every morphism in $\cat{C}$ has a (not necessarily unique) weak kernel.
    
    \item\label{item_adjoints} 
    The \textit{category of $\cat{C}$-modules}, denoted by $\rMod{\cat{C}}$, is the abelian category of all (covariant) additive functors $\fun{M}\colon \cat{C}^{\op}\rightarrow \Ab$. If $\cat{X}\subseteq \cat{C}$ is an additive subcategory, then there is a triplet of adjoint functors
    \begin{equation}\label{eqn_adjoints}
    \begin{tikzcd}[column sep=3cm]
        \rMod{\cat{C}} 
            \arrow{r}{\restr{(-)}{\cat{X}}}
        & \rMod{\cat{X}},
            \arrow[bend right]{l}[swap]{\adj{L}}
            \arrow[bend left]{l}{\adj{R}}
    \end{tikzcd}
    \end{equation}
    where the left adjoint $\adj{L}$ and the right adjoint $\adj{R}$ of the restriction functor $\restr{(-)}{\cat{X}}$ are fully faithful functors. Furthermore, $\adj{L}$ is right exact and $\adj{R}$ is left exact, while $\restr{(-)}{\cat{X}}$ is exact. See \cite[Props.~3.1 and 3.4]{Auslander-Rep-theory-of-Artin-algebras-I}.

    \item \label{item_modC} A $\cat{C}$-module $\fun{M}\colon \cat{C}^{\op}\rightarrow \Ab$ is \textit{finitely presented} if there is an exact sequence 
    \[
    \begin{tikzcd}
        \cat{C}(-,A)
            \arrow{r}
        &\cat{C}(-,B)
            \arrow{r}
        &\fun{M}
            \arrow{r}
        &0
    \end{tikzcd}
    \]
    for some objects $A,B\in\cat{C}$ (see \cite[p.~155]{Beligiannis-on-the-freyd-cats-of-additive-cats}). We denote  by $\rmod{\cat{C}}$ the subcategory of $\rMod{\cat{C}}$ consisting of the finitely presented $\cat{C}$-modules.
    In general, this is an additive category, but if $\cat{C}$ is also idempotent complete and has weak kernels, then $\rmod{\cat{C}}$ is abelian and the inclusion functor $\rmod{\cat{C}} \to \rMod{\cat{C}}$ is exact (see e.g.\  \cite[Sec.~III.2]{Auslander-representation-dimension-of-artin-algebras-QMC}). When $\cat{C}$ is a triangulated category, $\rmod{\cat{C}}$ is also known as the \emph{Freyd category of $\cat{C}$}; see \cite[Ch.~5]{Neeman-triangulated-categories}.
    
    \item\label{item_Yoneda} Suppose that $\cat{C}$ has split idempotents. Then the Yoneda embedding $\yoneda \colon  \cat{C}\rightarrow \rmod{\cat{C}}$ given by $C\mapsto \yoneda C\deff\cat{C}(-,C)$ is fully faithful and its values are projective objects in $\rmod{\cat{C}}$ (and in $\rMod{\cat{C}}$). 
    Since $\cat{C}$ is idempotent complete, any projective object in $\rmod{\cat{C}}$ is of  the form $\cat{C}(-,C)$ up to isomorphism. 
    See \cite[Prop.~2.2]{Auslander-Rep-theory-of-Artin-algebras-I}.
\end{enumerate}
\end{notation}

In this section, we work in the following setup.

\begin{setup}\label{setup_additive}
    Let $\cat{C}$ be a skeletally small, idempotent complete, additive category that has weak kernels. Furthermore, let $\cat{X}\subseteq \cat{C}$ be a contravariantly finite, additive subcategory. 
\end{setup}

Since $\cat{X}$ is contravariantly finite, it has weak kernels and, by \cref{notation_additive}\ref{item_modC}, we have that $\rmod{\cat{C}}$ and $\rmod{\cat{X}}$ are abelian categories and that their inclusions into $\rMod{\cat{C}}$ and $\rMod{\cat{X}}$ are exact functors.

\begin{lem}\label{lemma_restriction_finpres}
    The exact functor $\restr{(-)}{\cat{X}}\colon  \rMod{\cat{C}}\rightarrow \rMod{\cat{X}}$ restricts to an exact functor on finitely presented modules $\restr{(-)}{\cat{X}}\colon \rmod{\cat{C}}\rightarrow \rmod{\cat{X}}$.
\end{lem}

\begin{proof}
We first show that $\restr{\cat{C}(-,C)}{\cat{X}}$ is in $\rmod{\cat{X}}$ for any $C\in\cat{C}$. Since $\cat{X}$ is contravariantly finite in $\cat{C}$ and $\cat{C}$ has weak kernels, we can build a commutative diagram of the form
\[
\begin{tikzcd}
    X_{1}
        \arrow{dr}[swap]{x_{1}}
        \ar{rr}{c_{0}x_{1}}
    && X_0 
        \arrow{rr}{x_{0}}
    && C,\\
    &C_{0}
        \arrow{ur}[swap]{c_{0}}
    &&
\end{tikzcd}
\]
where $x_0,x_1$ are right $\cat{X}$-approximations and $c_0$ is a weak kernel of $x_0$. This induces an exact sequence
\[
\begin{tikzcd}[column sep=1.8cm]
    \cat{X}(-,X_{1})
        \arrow{r}{\cat{X}(-,c_{0}x_{1})}
    &\cat{X}(-,X_{0})
        \arrow{r}{\cat{X}(-,x_{0})} 
    &\restr{\cat{C}(-,C)}{\cat{X}}
        \arrow{r}
    &0,
\end{tikzcd}
\]
so that $\restr{\cat{C}(-,C)}{\cat{X}}\in\rmod{\cat{X}}$. 

Consider now an arbitrary object $\fun{M}\in\rmod{\cat{C}}$. By definition, there exist objects $A,B \in\cat{C}$ such that the sequence
$\cat{C}(-,A)\rightarrow\cat{C}(-,B)\rightarrow \fun{M}\rightarrow 0$ is exact. Applying $\restr{(-)}{\cat{X}}$ to this sequence, we obtain the exact sequence
\[
\begin{tikzcd}
\restr{\cat{C}(-,A)}{\cat{X}}
    \arrow{r}
&\restr{\cat{C}(-,B)}{\cat{X}}
    \arrow{r}
&\restr{\fun{M}}{\cat{X}}
    \arrow{r}
&0,
\end{tikzcd}
\]
where, by the above argument, the two leftmost objects are in $\rmod{\cat{X}}$ and hence $\restr{\fun{M}}{\cat{X}}\in\rmod{\cat{X}}$. Moreover, the exactness of the original functor implies the exactness of the restricted 
functor $\restr{(-)}{\cat{X}}\colon \rmod{\cat{C}}\rightarrow \rmod{\cat{X}}$.
\end{proof}

\begin{lem}\label{lemma_left_functor}
    The functor $\adj{L}$ from \eqref{eqn_adjoints} satisfies $\adj{L}(\cat{X}(-,X)) \cong\cat{C}(-,X)$ for any $X\in\cat{X}$.
\end{lem}

\begin{proof}
    Let $X\in\cat{X}$. For $\fun{G}\in\rMod{\cat{C}}$, we have 
    \begin{align*}
        (\rMod{\cat{C}})(\adj{L}(\cat{X}(-,X)), \fun{G})
            &\cong
        (\rMod{\cat{X}})(\cat{X}(-,X), \restr{\fun{G}}{\cat{X}})\\
        &\cong \restr{\fun{G}}{\cat{X}}(X)\\
        &= \fun{G}(X)\\
        &\cong (\rMod{\cat{C}})(\cat{C}(-,X), \fun{G}),
    \end{align*}
    where the first isomorphism holds because $\adj{L}$ is left adjoint to $\restr{(-)}{\cat{X}}$, and the second and fourth ones hold by Yoneda lemma (see e.g.\ \cite[Thm.~IV.6.1]{AssemSimsonSkowronski-Vol1}). The result follows by e.g.\ \cite[Cor.~IV.6.3]{AssemSimsonSkowronski-Vol1}, since the above is true for a general element $\fun{G}\in\rMod{\cat{C}}$. 
\end{proof}

\begin{lem}
\label{lemma_adjoint_restricted}
    The functors $\adj{L}$ and $\restr{(-)}{\cat{X}}$ from \eqref{eqn_adjoints} restrict to a pair of adjoint functors
    \[
    \begin{tikzcd}[column sep=2cm]
    \rmod{\cat{C}}
        \arrow{r}[swap]{\restr{(-)}{\cat{X}}}
    & \rmod{\cat{X}},
        \arrow[bend right]{l}[swap]{\adj{L}}
    \end{tikzcd}
    \]
    where $\adj{L}$ is full and faithful.
\end{lem}

\begin{proof}
    It is enough to see that the functors $\restr{(-)}{\cat{X}}$ and $\adj{L}$ from \eqref{eqn_adjoints} restrict to functors between $\rmod{\cat{C}}$ and $\rmod{\cat{X}}$. We already know the functor $\restr{(-)}{\cat{X}}$ does by \cref{lemma_restriction_finpres}.
    
    Thus, consider an element $\fun{M}\in\rmod{\cat{X}}$. By definition, there exists an exact sequence of the form 
    \[
    \begin{tikzcd}
        \cat{X}(-,X_1)
            \arrow{r}
        & \cat{X}(-,X_0)
            \arrow{r}
        & \fun{M}
            \arrow{r}
        & 0.
    \end{tikzcd}
    \]
    Since, as recalled in \cref{notation_additive}\ref{item_adjoints}, $\adj{L}$ is a right exact functor, it maps the above right exact sequence to a right exact sequence, which is isomorphic to one of the form
    \[
    \begin{tikzcd}
        \cat{C}(-,X_1)
            \arrow{r}
        & \cat{C}(-,X_0)
            \arrow{r}
        & \adj{L}(\fun{M})
            \arrow{r}
        & 0
    \end{tikzcd}
    \]
    by \cref{lemma_left_functor}. Hence, we have $\adj{L}(\fun{M})\in\rmod{\cat{C}}$.
\end{proof}

\begin{prop}\label{prop:restriction-right-exact-sequence-of-K0}
    The functor $\restr{(-)}{\cat{X}}\colon \rmod{\cat{C}}\rightarrow \rmod{\cat{X}}$ induces an exact sequence of Grothendieck groups
    \[
    \begin{tikzcd}[column sep=1.8cm]
        K_0(\Ker \restr{(-)}{\cat{X}})
            \arrow{r}{K_{0}(\iota)}
        &K_0(\rmod{\cat{C}})
            \arrow{r}{K_0(\restr{(-)}{\cat{X}})} 
        &K_0(\rmod{\cat{X}})
            \arrow{r}
        &0,
    \end{tikzcd}
    \]
    where $K_{0}(\iota)$ is the canonical morphism mapping $[\fun{M}]$ to $[\fun{M}]$.
    In particular, the group $\Ker K_0(\restr{(-)}{\cat{X}})$ is generated by $\Set{ [\fun{M}] | \fun{M}\in\rmod{\cat{C}} \text{ and } \restr{\fun{M}}{\cat{X}}=0 }$.
\end{prop}

\begin{proof}
    Consider the adjunction $\adj{L}\dashv \restr{(-)}{\cat{X}}$ from \cref{lemma_adjoint_restricted}, where $\adj{L}\colon \rmod{\cat{X}}\rightarrow \rmod{\cat{C}}$ is full and faithful. By \cite[Exam.~2.2.3, Prop.~2.2.11 and Rem.~2.2.12]{krause2021homological}, we have that $\Ker\restr{(-)}{\cat{X}}$ is a Serre subcategory of $\rmod{\cat{C}}$ and $\restr{(-)}{\cat{X}}$ induces an exact equivalence of categories
    \[
    \begin{tikzcd}
        \rmod{\cat{C}}/(\Ker\restr{(-)}{\cat{X}})
            \arrow{r}{\simeq} 
        &\rmod{\cat{X}}.
    \end{tikzcd}
    \]
    Then, by \cite[Cor.~VIII.5.5]{Bass-algebraic-k-theory-book}, we have an exact sequence of Grothendieck groups
    \[
    \begin{tikzcd}[column sep=1.8cm]
        K_0(\Ker \restr{(-)}{\cat{X}})
            \arrow{r}{K_{0}(\iota)}
        &K_0(\rmod{\cat{C}})
            \arrow{r}{K_0(\restr{(-)}{\cat{X}})} 
        &K_0(\rmod{\cat{X}})
            \arrow{r}
        &0,
    \end{tikzcd}
    \]
    where $K_{0}(\iota)$ is the 
    mapping given by $[\fun{M}]\mapsto [\fun{M}]$ induced by the canonical inclusion 
    $\iota\colon \Ker \restr{(-)}{\cat{X}} \into \rmod{\cat{C}}$. 
    Moreover, since
    \[
        \Ker \restr{(-)}{\cat{X}}
            = \Set{\fun{M}\in\rmod{\cat{C}} | \restr{\fun{M}}{\cat{X}}=0 },
    \]
    we have that $\Ker K_0(\restr{(-)}{\cat{X}})=\Im K_{0}(\iota)$ is generated by $\Set{ [\fun{M}] | \fun{M}\in\rmod{\cat{C}} \text{ and } \restr{\fun{M}}{\cat{X}}=0 }$.
\end{proof}

%%%%%%%%%%%%%%%%%%%%%%%%%%%%%%%%%%%%%%%%%%%%%%%%
%%%%%%%%%%%%%%%%%%%%%%%%%%%%%%%%%%%%%%%%%%%%%%%%
\section{The additivity formula with error term}
\label{sec:main-results}

In this section, we prove \cref{thm:F}, the additivity formula with error term for our index. The first key step is to prove \cref{thm:G}.

We fix the following setup throughout \cref{sec:main-results}, which is a special case of \cref{setup_additive}. 
Additionally, we note  \cref{setup:cluster-category} is a special case of \cref{setup:section3}. 
Recall from \cref{sec:conventions} that we assume an additive subcategory is closed under direct summands.

\begin{setup}\label{setup:section3}
    Let $(\cat{C},\sus,\triangle)$ be a skeletally small triangulated category with split idempotents. 
    We also suppose $\cat{X}\sse\cat{C}$ is a contravariantly finite, additive subcategory.
\end{setup}

We denote the extriangulated category corresponding to the triangulated category $(\cat{C},\sus,\triangle)$ by $(\cat{C},\BE,\fs)$ (see \cref{sec:triangulated-is-extriangulated}).

Recall that $\yoneda \colon \cat{C} \to \rmod{\cat{C}}$ denotes the Yoneda embedding (see \cref{notation_additive}\ref{item_Yoneda}). 
Under \cref{setup:section3}, we have that $\rmod{\cat{C}}$ is a Frobenius abelian category, where the projective-injective objects are precisely those of the form $\yoneda C$ for some $C\in\cat{C}$; see \cite[Cor.~5.1.23]{Neeman-triangulated-categories}. 
Using \cref{lemma_restriction_finpres}, we denote by $\yoneda _{\cat{X}}\colon\cat{C} \to \rmod{\cat{X}}$ the restricted Yoneda functor, which is the composition
\[
\begin{tikzcd}
\cat{C}
    \arrow{dr}[swap]{\yoneda }
    \arrow{rr}{\yoneda _{\cat{X}}}
    &
& \rmod{\cat{X}}. \\
& \rmod{\cat{C}}
    \arrow{ur}[swap]{\restr{(-)}{\cat{X}}}
&
\end{tikzcd}
\]
For each $X\in\cat{X}$, we have that $\yoneda _{\cat{X}}(X) = \restr{\cat{C}(-,X)}{\cat{X}} = \cat{X}(-,X)$ 
is a projective object, and just like in \cref{notation_additive}\ref{item_Yoneda} all projective objects in $\rmod{\cat{X}}$ arise in this way up to isomorphism. (Note that idempotents split in $\cat{X}$ because they do so in $\cat{C}$ and $\cat{X}$ is closed under direct summands.)

%%%%%%%%%%%%%%%%%%%%%%%%%%%%%%%%%%%%%%%%%%%%%%%%
\subsection{Defining \texorpdfstring{$\theta_{\cat{C}}$}{thetaC}}

The aim of this subsection is to produce a group homomorphism $\theta_{\cat{C}}\colon K_{0}(\rmod{\cat{C}})\to K^{\sp}_{0}(\cat{C})
$, which we do via three lemmas, and prove \cref{thm:G}.
As $\cat{C}$ is a triangulated category, we see that $\yoneda $ and  $\yoneda _{\cat{X}}$ are both cohomological functors, taking triangles in $\cat{C}$ to long exact sequences of functors; see e.g.\ \cite[Lem.~5.1.17]{Neeman-triangulated-categories}. Thus we immediately have the first lemma.

\begin{lem}
\label{lem:lemma2}
Each triangle 
\begin{equation}\label{eqn:triangle2}
\begin{tikzcd}
A \arrow{r}{a} & B \arrow{r}{b} & C	 \arrow{r}{c} & \sus A.
\end{tikzcd}
\end{equation}
in $\cat{C}$ induces an exact sequence
\begin{equation}\label{eqn:2b}
    \begin{tikzcd}[column sep=1.5cm]
        \yoneda A 
            \arrow{r}{\yoneda a}
        & \yoneda B 
            \arrow{r}{\yoneda b}
        & \yoneda C 
            \arrow{r}{\im\yoneda c}
        & \Im \yoneda c
            \arrow{r}
        & 0
    \end{tikzcd}
\end{equation}
in $\rmod{\cat{C}}$. 
Moreover, any object $\fun{M}\in\rmod{\cat{C}}$ is isomorphic to $\Im \yoneda c$ for some morphism $c\colon C\to \sus A$ in a triangle 
\eqref{eqn:triangle2}.
\end{lem}

\begin{proof}
It is clear that \eqref{eqn:2b} is an exact sequence in $\rMod{\cat{C}}$ since $\yoneda $ is cohomological. 
In particular, it demonstrates that $\Im \yoneda c$ is finitely presented, and hence \eqref{eqn:2b} is exact in $\rmod{\cat{C}}$.

Now let $\fun{M}\in\rmod{\cat{C}}$ be arbitrary. Then there is an exact sequence
\[
\begin{tikzcd}
    \yoneda B
        \arrow{r}{\yoneda b}
    &\yoneda C
        \arrow{r}{}
    &\fun{M}
        \arrow{r}{}
    &0.
\end{tikzcd}
\]
The morphism $b\colon B \to C$ is part of a triangle 
\eqref{eqn:triangle2}, 
so we also have the exact sequence \eqref{eqn:2b}. Since both $\fun{M}$ and $\Im \yoneda c$ are cokernels of $\yoneda b$, we have $\fun{M} \cong \Im \yoneda c$.
\end{proof}

The proof of the following lemma is inspired by the proof of \cite[Lem.~3.9]{CGMZ}.
The proof uses that since $\sus\colon \cat{C}\to \cat{C}$ is a triangulated automorphism of $\cat{C}$, it induces an exact automorphism
$\bbsus^{-1}\colon \rMod{\cat{C}} \to \rMod{\cat{C}}$ given by precomposition with $\sus$. That is, 
for an object $\fun{M}\in\rMod{\cat{C}}$, we have 
$\bbsus^{-1}\fun{M} \deff \fun{M}\circ\sus$ and, 
for a natural transformation
$\alpha\colon \fun{M} \to \fun{N}$, we have 
$(\bbsus^{-1}\alpha)_{C} \deff \alpha_{\sus C}$ for each $C\in\cat{C}$.
See \cite[Sec.~2]{Krause-cohomological-length-functions} or \cite[Sec.~3.1]{CGMZ}.
Furthermore, we observe here that $\yoneda (\sus^{-1}C)$ is naturally isomorphic to $\bbsus^{-1}(\yoneda C)$. 
The automorphism $\bbsus^{-1}$ has inverse $\bbsus$ which is precomposition by $\sus^{-1}$, and we have $\yoneda (\sus C) \cong \bbsus(\yoneda C)$.
The notation $\bbsus$ is borrowed from Nkansah \cite{Nkansah}.

\begin{lem}
\label{lem:lemma4}
If 
$\begin{tikzcd}[column sep=0.5cm,cramped]
A \arrow{r}{a} & B \arrow{r}{b} & C	 \arrow{r}{c} & \sus A
\end{tikzcd}$
and 
$\begin{tikzcd}[column sep=0.5cm,cramped]
A' \arrow{r}{a'} & B' \arrow{r}{b'} & C'	 \arrow{r}{c'} & \sus A'
\end{tikzcd}$
are triangles in $\cat{C}$ with an isomorphism $\phi\colon \Im \yoneda c \overset{\cong}{\longrightarrow} \Im \yoneda c'$, then in $K^{\sp}_{0}(\cat{C})$ we have  
\[
[A]^{\sp} - [B]^{\sp} + [C]^{\sp} 
    = [A']^{\sp} - [B']^{\sp} + [C']^{\sp}.
\]
\end{lem}

\begin{proof}
We claim that there is a commutative diagram in $\rmod{\cat{C}}$ with exact rows as follows.
\begin{equation}\label{eqn:4c}
\begin{tikzcd}[column sep=0.5cm]
% row 1
\yoneda A
    \arrow{rr}{\yoneda a}
    \arrow{dd}{\yoneda e}
&{}
&\yoneda B
    \arrow{rr}{\yoneda b}
    \arrow{dd}{\yoneda f}
&{}
&\yoneda C
    \arrow{rr}[xshift=-22pt]{\yoneda c}
    \arrow[two heads]{dr}{\beta}
    \arrow{dd}{\yoneda g}
&{}
&\yoneda \sus A
\arrow{dd}{\yoneda \sus e}
\\
% row 2
{}
&{}
&{}
&{}
&{}
&\Im\yoneda c
    \arrow[hook]{ur}{\alpha}
    \arrow{dd}[yshift=-15pt]{\cong}[yshift=-15pt,swap]{\phi}
&{}
\\
% row 3
\yoneda A'
    \arrow{rr}{\yoneda a'}
&{}
&\yoneda B'
    \arrow{rr}{\yoneda b'}
&{}
&\yoneda C'
    \arrow[-]{r}{\yoneda c'}
    \arrow[two heads]{dr}[swap]{\beta'}
&{}
    \arrow{r}
&\yoneda \sus A'
\\
% row 4
{}
&{}
&{}
&{}
&{}
&\Im\yoneda c'
    \arrow[hook]{ur}[swap]{\alpha'}
&{}
\end{tikzcd}
\end{equation}
Indeed, the exact rows come from $\yoneda $ being cohomological. 
Since $\yoneda \sus A$ is injective and $\yoneda$ is full, there is a morphism 
$\sus e\colon \sus A\to \sus A'$ such that 
$(\yoneda \sus e)\alpha = \alpha' \phi$.
Similarly, $\yoneda C$ being projective implies there is a morphism $g\colon C\to C'$ such that
$\phi\beta = \beta'\yoneda g$ (again using that $\yoneda$ is full). 
As $\yoneda $ is faithful, we see that 
$(\sus e)c = c' g$. Axiom (TR3) for triangulated categories (see e.g.\ \cite[Def.~1.1.2]{Neeman-triangulated-categories}) yields a morphism $f\colon B \to B'$ so that 
\begin{equation}\label{eqn:4d}
\begin{tikzcd}
A
    \arrow{r}{a}
    \arrow{d}{e}
& B
    \arrow{r}{b}
    \arrow[dotted]{d}{f}
& C 
    \arrow{r}{c}
    \arrow{d}{g}
& \sus A
    \arrow{d}{\sus e}
\\
A'
    \arrow{r}{a'}
& B'
    \arrow{r}{b'}
& C'
    \arrow{r}{c'}
& \sus A'
\end{tikzcd}
\end{equation}
is a morphism of triangles. Applying $\yoneda $ to \eqref{eqn:4d} gives \eqref{eqn:4c}.

Using that $\bbsus^{-1}$ is an exact automorphism and $\bbsus^{-1}(\yoneda X) \cong \yoneda (\sus^{-1} X)$, we can augment the left hand side of \eqref{eqn:4c} to obtain the following isomorphism of $3$-extensions in $\rmod{\cat{C}}$. 
\begin{equation}\label{eqn:4e}
\begin{tikzcd}
\bbsus^{-1}\Im \yoneda c 
    \arrow[hook]{r}
    \arrow{d}{\bbsus^{-1}\phi}[swap]{\cong}
&\yoneda A
    \arrow{r}{\yoneda a}
    \arrow{d}{\yoneda e}
&\yoneda B
    \arrow{r}{\yoneda b}
    \arrow{d}{\yoneda f}
&\yoneda C
    \arrow[two heads]{r}{\beta}
    \arrow{d}{\yoneda g}
& \Im \yoneda c
    \arrow{d}{\phi}[swap]{\cong}
\\
\bbsus^{-1}\Im \yoneda c'
    \arrow[hook]{r}
&\yoneda A'
    \arrow{r}{\yoneda a'}
&\yoneda B'
    \arrow{r}{\yoneda b'}
&\yoneda C'
    \arrow[two heads]{r}{\beta'}
& \Im \yoneda c'
\end{tikzcd}
\end{equation}
Thus, by \cite[Lem.~A.2]{Krause-cohomological-length-functions}, we have 
$\yoneda C \oplus \yoneda B' \oplus \yoneda A \cong \yoneda C' \oplus \yoneda B \oplus \yoneda A'$ in $\rmod{\cat{C}}$. 
Since $\yoneda $ is additive and fully faithful, we have 
$C\oplus B' \oplus A \cong C' \oplus B \oplus A'$ in $\cat{C}$, and hence
$[C]^{\sp} + [B']^{\sp} + [A]^{\sp} = [C']^{\sp} + [B]^{\sp} + [A']^{\sp}$ in $K^{\sp}_{0}(\cat{C})$. Rearranging this identity finishes the proof.
\end{proof}

The following lemma is due to the proof of \cite[Thm.~3.10]{CGMZ}, and we reproduce a version of the proof therein.

\begin{lem}
\label{lem:lemma3}
Suppose 
$\begin{tikzcd}[column sep=0.5cm]
    0
        \arrow{r}
    & \fun{M'}
        \arrow{r}{\alpha}
    & \fun{M}
        \arrow{r}{\beta}
    & \fun{M''}
        \arrow{r}
    & 0
\end{tikzcd}$
is a short exact sequence in $\rmod{\cat{C}}$. 
Then there exist triangles
\begin{align}
    \begin{tikzcd}[ampersand replacement=\&]
    A' \arrow{r}{a'} \& B' \arrow{r}{b'} \& C'	 \arrow{r}{c'} \& \sus A',
    \end{tikzcd} \label{eqn:3ai}\\
    \begin{tikzcd}[ampersand replacement=\&]
    A'' \arrow{r}{a''} \& B'' \arrow{r}{b''} \& C''	 \arrow{r}{c''} \& \sus A'' \label{eqn:3aii}
    \end{tikzcd}
\end{align}
in $\cat{C}$, satisfying 
$\fun{M'} \cong \Im \yoneda c'$
and 
$\fun{M''} \cong \Im \yoneda c''$, 
such that there is also a triangle 
\begin{equation}\label{eqn:3b}
\begin{tikzcd}[ampersand replacement=\&]
    A'\oplus A'' \arrow{r}{} \& B'\oplus B'' \arrow{r}{} \& C'\oplus C''	 \arrow{r}{c} \& \sus A'\oplus \sus A''
\end{tikzcd}
\end{equation}
with $\fun{M} \cong \Im \yoneda c$.
\end{lem}

\begin{proof}
\cref{lem:lemma2} provides triangles \eqref{eqn:3ai} and \eqref{eqn:3aii} that, respectively, induce exact sequences
\begin{align}
    \begin{tikzcd}[column sep=1.5cm, ampersand replacement=\&]
        \yoneda A' 
            \arrow{r}{\yoneda a'}
        \& \yoneda B'
            \arrow{r}{\yoneda b'}
        \& \yoneda C'
            \arrow{r}{\im\yoneda c'}
        \& \fun{M}'
            \arrow{r}
        \& 0,
    \end{tikzcd}\label{eqn:3ci}\\
    \begin{tikzcd}[column sep=1.5cm, ampersand replacement=\&]
        \yoneda A'' 
            \arrow{r}{\yoneda a''}
        \& \yoneda B'' 
            \arrow{r}{\yoneda b''}
        \& \yoneda C'' 
            \arrow{r}{\im\yoneda c''}
        \& \fun{M}''
            \arrow{r}
        \& 0
    \end{tikzcd}\label{eqn:3cii}
\end{align}
in $\rmod{\cat{C}}$. 
Objects of the form $\yoneda X$ for $X\in\cat{C}$ are projective in $\rmod{\cat{C}}$, so the Horseshoe Lemma (see e.g.\ \cite[Lem.~IX.7.8]{Aluffi-Chapter0}) yields a commutative diagram
\begin{equation}\label{eqn:3d}
\begin{tikzcd}
 0  \arrow{d}{}& 0  \arrow{d}{}& 0  \arrow{d}{}&   \\
  \yoneda B' \arrow{d}{\begin{psmallmatrix}\id{\yoneda B'} \\ 0\end{psmallmatrix}}\arrow{r}{\yoneda b'}
  & \yoneda C' \arrow{d}{\begin{psmallmatrix}\id{\yoneda C'} \\ 0\end{psmallmatrix}} \arrow{r}{\im\yoneda c'}
  & \fun{M'}  \arrow{d}{\alpha}\arrow{r}{}
  &  0 \\
  \yoneda B' \oplus \yoneda B'' 
  	\arrow{d}{\begin{psmallmatrix}0 \amph \id{\yoneda B''}\end{psmallmatrix}}
	\arrow{r}{\yoneda b}
  & \yoneda C' \oplus \yoneda C''  
  	\arrow{d}{\begin{psmallmatrix}0 \amph \id{\yoneda C''}\end{psmallmatrix}}
	\arrow{r}{}
  & \fun{M}  
  	\arrow{d}{\beta}\arrow{r}{}
  &  0 \\
  \yoneda B'' \arrow{d}{}\arrow{r}{\yoneda b''}
  & \yoneda C''  \arrow{d}{}\arrow{r}{\im\yoneda c''}
  & \fun{M''}  \arrow{d}{}\arrow{r}{}
  &  0 \\
 0  & 0  & 0  &  
\end{tikzcd}
\end{equation}
in $\rmod{\cat{C}}$ with exact rows and columns. 
Note that the morphism $b\colon B'\oplus B'' \to C' \oplus C''$ is obtained using that $\yoneda $ is fully faithful and additive. 
In particular, the two leftmost columns of \eqref{eqn:3d} are induced from the following morphism of split triangles in $\cat{C}$.
\begin{equation}\label{eqn:3e}
\begin{tikzcd}
B' \arrow{d}{\begin{psmallmatrix}\id{B'} \\ 0\end{psmallmatrix}}\arrow{r}{b'}& C' \arrow{d}{\begin{psmallmatrix}\id{C'} \\ 0\end{psmallmatrix}}\\
B' \oplus B'' \arrow{d}{\begin{psmallmatrix}0 \amph \id{B''}\end{psmallmatrix}}\arrow{r}{b}& C' \oplus C'' \arrow{d}{\begin{psmallmatrix}0 \amph \id{C''}\end{psmallmatrix}}\\
B''  \arrow{d}{0}\arrow{r}{b''}& C'' \arrow{d}{0}\\
\sus B' \arrow{r}{\sus b'}& \sus C'
\end{tikzcd}
\end{equation}
Thus, by the dual of \cite[Lem.~2.6]{May-the-additivity-of-traces-in-triangulated-categories},
we can complete \eqref{eqn:3e} to the commutative diagram
\begin{equation}\label{eqn:3f}
\begin{tikzcd}
A'
	\arrow{r}{a'}
	\arrow{d}{}
&B' 
	\arrow{d}{\begin{psmallmatrix}\id{B'} \\ 0\end{psmallmatrix}}
	\arrow{r}{b'}
& C' 
	\arrow{d}{\begin{psmallmatrix}\id{C'} \\ 0\end{psmallmatrix}}\\
A
	\arrow{r}{}
	\arrow{d}{}
&B' \oplus B'' 
	\arrow{d}{\begin{psmallmatrix}0 \amph \id{B''}\end{psmallmatrix}}
	\arrow{r}{b}
& C' \oplus C'' 
	\arrow{d}{\begin{psmallmatrix}0 \amph \id{C''}\end{psmallmatrix}}\\
A''
	\arrow{r}{a''}
	\arrow{d}{}
&B''  
	\arrow{d}{0}
	\arrow{r}{b''}
& C'' 
	\arrow{d}{0}\\
\sus A'
	\arrow{r}{\sus a'}
&\sus B' 
	\arrow{r}{\sus b'}
& \sus C' 
\end{tikzcd}
\end{equation}
in which each row is an $\fs$-conflation and each column is a triangle. 
In particular, the second row of \eqref{eqn:3f} is part of a triangle
\begin{equation}\label{eqn:3g}
\begin{tikzcd}
    A \arrow{r}{} & B' \oplus B''\arrow{r}{b} & C'\oplus C'' \arrow{r}{c} & \sus A.
    \end{tikzcd}
\end{equation}
We claim that \eqref{eqn:3g} can be used as triangle \eqref{eqn:3b}. 
Notice that $\fun{M}\cong \im\yoneda c$ by the exactness of the middle row of  \eqref{eqn:3d}.
Thus, it suffices to show that $A \cong A' \oplus A''$.

Applying $\yoneda $ to the top three rows of \eqref{eqn:3f} yields a commutative diagram in $\rmod{\cat{C}}$ consisting of three exact rows, in which the two rightmost columns are short exact. Splitting this diagram up into short exact sequences yields the following commutative diagram (that, for reference, has seven columns).
\begin{equation}\label{eqn:3h}
\begin{tikzcd}[column sep=0.8cm, scale cd=0.95]
% row 1
{}
&\yoneda A'
	\arrow{rr}[xshift=-30pt]{\yoneda a'}
	\arrow{dd}{}
	\arrow[two heads]{dr}{}
&{}
&\yoneda B' 
	\arrow{rr}[xshift=-24pt]{\yoneda b'}
	\arrow{dd}{\begin{psmallmatrix}\id{\yoneda B'} \\ 0\end{psmallmatrix}}
	\arrow[two heads]{dr}{}
&{}
& \yoneda C' 
	\arrow{dd}{\begin{psmallmatrix}\id{\yoneda C'} \\ 0\end{psmallmatrix}}
	\arrow[two heads]{dr}[xshift=-13pt,yshift=5pt]{\im\yoneda c'}
&{}
\\
% row 2
\bbsus^{-1}\fun{M'}
	\arrow[hook]{ur}{}
	\arrow{dd}{\bbsus^{-1}\alpha}
&{}
&\fun{L'}
	\arrow[hook]{ur}{}
	\arrow[crossing over]{dd}{}
&{}
&\fun{K'}
	\arrow[hook]{ur}{}
	\arrow[crossing over]{dd}{}
&{}
&\fun{M'}
	\arrow[hook]{dd}{\alpha}
\\
% row 3
{}
&\yoneda A
	\arrow[-]{r}{}
	\arrow{dd}{}
	\arrow[two heads]{dr}{}
&{}
        \arrow{r}{}
&\yoneda B' \oplus \yoneda B'' 
	\arrow[-]{r}[xshift=-5pt]{\yoneda b}
	\arrow{dd}[yshift=-10pt]{\begin{psmallmatrix}0 \amph \id{\yoneda B''}\end{psmallmatrix}}
	\arrow[two heads]{dr}{}
&{}
        \arrow{r}
& \yoneda C' \oplus \yoneda C'' 
	\arrow{dd}[yshift=-10pt]{\begin{psmallmatrix}0 \amph \id{\yoneda C''}\end{psmallmatrix}}
	\arrow[two heads]{dr}{}
&{}
\\
% row 4
\bbsus^{-1}\fun{M}
	\arrow[hook]{ur}{}
	\arrow{dd}{\bbsus^{-1}\beta}
&{}
&\fun{L}
	\arrow[hook]{ur}{}
	\arrow[crossing over]{dd}{}
&{}
&\fun{K}
	\arrow[hook]{ur}{}
	\arrow[crossing over]{dd}{}
&{}
&\fun{M}
	\arrow[two heads]{dd}{\beta}
\\
% row 5
{}
&\yoneda A''
	\arrow[-]{r}[xshift=0pt]{\yoneda a''}
	\arrow[two heads]{dr}{}
&{}
        \arrow{r}
&\yoneda B''  
	\arrow[-]{r}[xshift=6pt]{\yoneda b''}
	\arrow[two heads]{dr}{}
&{}
        \arrow{r}
& \yoneda C'' 
	\arrow{dr}[xshift=-13pt,yshift=5pt]{\im\yoneda c''}
&{}
\\
% row 6
\bbsus^{-1}\fun{M''}
	\arrow[hook]{ur}{}
&{}
&\fun{L''}
	\arrow[hook]{ur}{}
&{}
&\fun{K''}
	\arrow[hook]{ur}{}
&{}
&\fun{M''}
\end{tikzcd}
\end{equation}

We claim that the second column of \eqref{eqn:3h} is short exact. This is shown via multiple applications of the Nine Lemma (see e.g.\ \cite[Exer.~III.7.15]{Aluffi-Chapter0} and, in particular, \cite[Exer.~III.7.16]{Aluffi-Chapter0}).
Since columns six and seven are short exact, column five is short exact. 
Since columns four and five are short exact, so too is column three. 
Lastly, to see that column two is short exact, first note that column one is short exact because column seven is by assumption and 
$\bbsus^{-1}$ is exact. 
Second, as  
\begin{equation}\label{eqn:3i}
\begin{tikzcd}
A' \arrow{r} & A \arrow{r}& A''
\end{tikzcd}
\end{equation}
is an $\fs$-conflation, we have that the second column of \eqref{eqn:3h} is a complex. 
Thus, columns one and three being short exact imply column two is short exact by \cite[Exer.~III.7.16]{Aluffi-Chapter0}.

Since $\yoneda A''$ is projective, column two of \eqref{eqn:3h} splits. Moreover, as $\yoneda $ is fully faithful, we see that \eqref{eqn:3i} is a split $\fs$-conflation (i.e.\ a split short exact sequence) and the proof is complete.
\end{proof}

\begin{thm}[=\cref{thm:G}]
\label{thm:theorem5}
There is a group homomorphism 
$\theta_{\cat{C}}\colon K_{0}(\rmod{\cat{C}})\to K^{\sp}_{0}(\cat{C})
$
defined by 
$\theta_{\cat{C}}([\Im\yoneda c])
    = [A]^{\sp} - [B]^{\sp} + [C]^{\sp}$
where there is a triangle \eqref{eqn:triangle2} in $\cat{C}$.
\end{thm}

\begin{proof}
Let $\Iso(\rmod{\cat{C}})$ denote the set of isomorphism classes of objects in $\rmod{\cat{C}}$. 
By \cref{lem:lemma2} and \cref{lem:lemma4}, the assignment 
sending the isomorphism class of 
$\Im\yoneda c$ to $[A]^{\sp} - [B]^{\sp} + [C]^{\sp} \in K^{\sp}_{0}(\cat{C})$ 
induces a well-defined set-function 
$\Iso(\rmod{\cat{C}}) \to K^{\sp}_{0}(\cat{C})$. 
\cref{lem:lemma3} ensures this induces 
the homomorphism $\theta_{\cat{C}}$ as claimed.
\end{proof}

%%%%%%%%%%%%%%%%%%%%%%%%%%%%%%%%%%%%%%%%%%%%%%%%
\subsection{Defining \texorpdfstring{$\theta_{\cat{X}}$}{thetaX}}

The following theorem (which is \cref{thm:F}) is a strict generalisation of \cite[Thm.~A]{JorgensenShah-index}, removing the restriction that $\cat{X}$ is rigid in the triangulated category $\cat{C}$. 
The relative extriangulated category $(\cat{C},\BE_{\cat{X}},\fs_{\cat{X}})$ was defined in \cref{sec:relative-theory} (see also \cref{sec:triangulated-is-extriangulated}). 
Recall from \cref{def:A} that, for an object $C\in\cat{C}$, the class $[C]_{\cat{X}}\in K_{0}(\cat{C},\BE_{\cat{X}},\fs_{\cat{X}})$ is the index with respect to $\cat{X}$ of $C$, denoted by $\indxx{ \cat{ X }}(C)$.

\begin{thm}[=\cref{thm:F}]
\label{thm:theorem9}
There is a group homomorphism 
$\theta_{\cat{X}}\colon K_{0}(\rmod{\cat{X}})\to K_{0}(\cat{C},\BE_{\cat{X}},\fs_{\cat{X}})
$
such that if \eqref{eqn:triangle2} is a triangle in $\cat{C}$, then 
$$\theta_{\cat{X}}([\Im\yoneda _{\cat{X}}c])
    = \indxx{ \cat{ X }}(A) - \indxx{ \cat{ X }}(B) + \indxx{ \cat{ X }}(C).$$
\end{thm}

\begin{proof}
Recall the homomorphism $\theta_{\cat{C}}\colon K_{0}(\rmod{\cat{C}})\to K^{\sp}_{0}(\cat{C})$ from \cref{thm:theorem5}
and 
the canonical surjection 
$\pi_{\cat{X}}\colon K^{\sp}_{0}(\cat{C}) \to K_{0}(\cat{C},\BE_{\cat{X}},\fs_{\cat{X}})$ from \eqref{eqn:all-the-pi}. 
In light of \cref{prop:restriction-right-exact-sequence-of-K0}, it suffices to show that the homomorphism $\pi_{\cat{X}}\theta_{\cat{C}}$ vanishes on 
$\Im K_{0}(\iota)$ to produce a homomorphism $\theta_{\cat{X}}\colon K_{0}(\rmod{\cat{X}})\to K_{0}(\cat{C},\BE_{\cat{X}},\fs_{\cat{X}})
$ 
making the following diagram commute.
\begin{equation}\label{eqn:9a}
\begin{tikzcd}[column sep=1.5cm]
K_{0}(\Ker \restr{(-)}{\cat{X}} )
    \arrow{r}{K_{0}(\iota)}
& K_{0}(\rmod{\cat{C}})
    \arrow{r}{K_0(\restr{(-)}{\cat{X}})}
    \arrow{d}[swap]{\theta_{\cat{C}}}
& K_{0}(\rmod{\cat{X}})
    \arrow{r}{}
    \arrow[dotted]{ddl}{\theta_{\cat{X}}}
& 0
\\
{}
&K^{\sp}_{0}(\cat{C})
    \arrow{d}[swap]{\pi_{\cat{X}}}
&{}
&{}
\\
{}
& K_{0}(\cat{C},\BE_{\cat{X}},\fs_{\cat{X}})
&{}
&{}
\end{tikzcd}
\end{equation}

\cref{prop:restriction-right-exact-sequence-of-K0} also implies it is enough to show that $\pi_{\cat{X}}\theta_{\cat{C}}K_{0}(\iota) ([\fun{M}])=0$ for $\fun{M}\in\rmod{\cat{C}}$ with $\restr{\fun{M}}{\cat{X}} = 0$.
Given such an $\fun{M}$, by \cref{lem:lemma2} we can pick a triangle 
$\begin{tikzcd}[column sep=0.4cm,cramped]
    A \arrow{r}{a}& B\arrow{r}{b} & C \arrow{r}{c}& \sus A
\end{tikzcd}$ 
with $\fun{M}\cong \Im\yoneda c$. 
Note that we have 
\[
\Im (\yoneda _{\cat{X}}c)
    = \Im (\restr{(\yoneda c)}{\cat{X}})
    = \restr{(\Im\yoneda c)}{\cat{X}} 
    \cong  \restr{\fun{M}}{\cat{X}}
    = 0,
\]
using the exactness of $\restr{(-)}{\cat{X}}$ (see \cref{lemma_restriction_finpres}). 
So, from \eqref{eqn:relative-yoneda} we see that 
$\begin{tikzcd}[column sep=0.5cm,cramped]
    A \arrow{r}{a}& B\arrow{r}{b} & C \arrow[dashed]{r}{c}& {}
\end{tikzcd}$
is in fact an $\fs_{\cat{X}}$-triangle. 
In particular, this implies 
\begin{equation}\label{eqn:9b}
    [A]_{\cat{X}} - [B]_{\cat{X}} + [C]_{\cat{X}} = 0
\end{equation}
in $K_{0}(\cat{C},\BE_{\cat{X}},\fs_{\cat{X}})$. 
Moreover, we have 
\begin{align*}
\pi_{\cat{X}}\theta_{\cat{C}}K_{0}(\iota) ([\fun{M}])
    &= \pi_{\cat{X}}\theta_{\cat{C}} ([\Im\yoneda c])\\
    &= \pi_{\cat{X}}([A]^{\sp} - [B]^{\sp} + [C]^{\sp}) && \text{by \cref{thm:theorem5}}\\
    &= [A]_{\cat{X}} - [B]_{\cat{X}} + [C]_{\cat{X}}\\
    &= 0 &&\text{by \eqref{eqn:9b}}.
\end{align*}

Lastly, checking that 
$$\theta_{\cat{X}}([\Im\yoneda _{\cat{X}}c])
    = [A]_{\cat{X}} - [B]_{\cat{X}} + [C]_{\cat{X}} = \indxx{ \cat{ X }}(A) - \indxx{ \cat{ X }}(B) + \indxx{ \cat{ X }}(C)$$
for an arbitrary triangle 
\eqref{eqn:triangle2} in $\cat{C}$ is straightforward using the commutativity of \eqref{eqn:9a} and the exactness of $\restr{(-)}{\cat{X}}$. 
\end{proof}

%%%%%%%%%%%%%%%%%%%%%%%%%%%%%%%%%%%%%%%%%%%%%%%%
%%%%%%%%%%%%%%%%%%%%%%%%%%%%%%%%%%%%%%%%%%%%%%%%

\begin{acknowledgements}
We would like to thank 
\.{I}lke \c{C}anak\c{c}\i, 
Sofia Franchini, 
Matthew Pressland 
and 
Emine Y{\i}ld{\i}r{\i}m 
for helpful discussions. 
We are also grateful to an anonymous referee for their comments on a previous version.

FF is supported by the EPSRC Programme Grant EP/W007509/1. PJ and AS are supported by a DNRF Chair from the Danish National Research Foundation (grant DNRF156), by a Research Project 2 from the Independent Research Fund Denmark, and by the Aarhus University Research Foundation (grant AUFF-F-2020-7-16).
\end{acknowledgements}

%%%%%%%%%%%%%%%%%%%%%%%%%%%%%%%%%%%%%%%%%%%%%%%%
%%%%%%%%%%%%%%%%%%%%%%%%%%%%%%%%%%%%%%%%%%%%%%%%

%%%%%%%%%%%%%%%%%%%%%%%%%%%%%
\end{document}